\documentclass[11pt]{amsart}

\usepackage{hyperref}
\usepackage{amsfonts}
\usepackage{amssymb}
\usepackage{amsmath}
\usepackage{amsthm}
\usepackage{latexsym}
\usepackage{color}
\usepackage{amscd}
\usepackage[all,cmtip]{xy}
\usepackage{graphicx}
\usepackage{enumitem} 
\usepackage[margin=1.6in]{geometry}
\usepackage{esint} 
\usepackage{epsfig}

\allowdisplaybreaks

\usepackage{fullpage}

\numberwithin{equation}{section}
\newcounter{exercise}

\newtheorem{theorem}{Theorem}
\newtheorem{proposition}[theorem]{Proposition}
\newtheorem{lemma}[theorem]{Lemma}

\newtheorem{prop}[theorem]{Proposition}
\newtheorem{cor}[theorem]{Corollary}

\theoremstyle{definition}
\newtheorem{definition}[theorem]{Definition}

\theoremstyle{remark}

\newtheorem{rmk}[theorem]{Remark}

\newcommand\R{{\mathbb R}}

\newcommand\Q{{\mathbb Q}}

\newcommand{\cA}{\mathcal A}
\newcommand{\cB}{\mathcal B}

\newcommand{\cF}{\mathcal F}
\newcommand{\cG}{\mathcal G}

\def\eps{{\varepsilon}}

\newcommand{\dd}{{\, \mathrm d}}

\newcommand{\ds}{\displaystyle}
\newcommand{\fa}{\forall \,}

\let\oldmarginpar\marginpar
\renewcommand\marginpar[1]{\-\oldmarginpar[\raggedleft\footnotesize #1]%
{\raggedright\footnotesize #1}}

\def\P{\mathbb{P}}
\def\Q{\mathbb{Q}}
\def\E{\mathbb{E}}

\newcommand{\indic}[1]{\mathbf{1}_{\{#1\}}}
\newcommand{\BB}[3]{\mathbb{P}^{#1 \stackrel{#3}{\rightarrow} #2} }

\def\signnb{\bigskip \begin{center} {\sc Nathana\"el
      Berestycki\par\vspace{3mm}
      University of Cambridge\par
      Statistical Laboratory, DPMMS,\par
      Wilberforce road, Cambridge CB3 0WB, UK
      \par\vspace{3mm} e-mail:}
    \tt{N.Berestycki@statslab.cam.ac.uk} \end{center}}

\def\signcm{\bigskip \begin{center} {\sc Cl\'ement
      Mouhot\par\vspace{3mm}
      University of Cambridge\par
      DPMMS, Centre for Mathematical Sciences\par
      Wilberforce road, Cambridge CB3 0WB, UK
      \par\vspace{3mm} e-mail:}
    \tt{C.Mouhot@dpmms.cam.ac.uk} \end{center}}

\def\signgr{\bigskip \begin{center} {\sc Ga\"el
      Raoul\par\vspace{3mm}
      \'Ecole Polytechnique\par
      CMAP, Route de Saclay, 91128 Palaiseau, France\par
      \par\vspace{3mm} e-mail:}
    \tt{raoul@cmap.polytechnique.fr} \end{center}}
\begin{document}

\title[Self-accelerating fronts in reaction-diffusion
theory]{Existence of self-accelerating fronts for a non-local
  reaction-diffusion equations}

\author{N. Berestycki, C. Mouhot, G. Raoul}

\date{\today}

\maketitle

\begin{abstract} We describe the accelerated propagation wave arising from a non-local reaction-diffusion equation. This equation originates from an ecological problem, where accelerated biological invasions have been documented. The analysis is based on the comparison of this model with a related local equation, and on the analysis of the dynamics of the solutions of this second model thanks to probabilistic methods.
\end{abstract} 

\tableofcontents

\section{Introduction and results}
\label{sec:introduction}

Biological invasions happen when a species recently introduced in a location succeeds to establish and to spread in this new environment. These introduction are usually either a consequence of human transportation systems \cite{Carlton}, or a consequence of the climate change \cite{Kovats}. Biological invasions are occuring at an unprecedented rate \cite{Hulme}, and have an important impact on e.g. biodiversity \cite{Sakai} and human well-being \cite{Pejchar,Juliano}. Predicting the dynamics of those invasion is an issue, that requires (among other approaches) the development of new mathematical methods and results \cite{Clark, Kolar}.


In this study, we are interested in a particular phenomenon that may happen during biological invasions \cite{Thomas2,Mack,Edmonds}: the dispersion of the individuals increases during the invasion. As a result, 
the speed of the invasive front increases, and often keeps accelerating as long as the invasion progresses \cite{Mack}. The best documented case is a biological invasion of Cane Toads in Australia \cite{Phillips,Urban}. The amphibians have been introduced in Australia in 1935 as a (failed) attempt to control beetles populations in cane plantations. Since then, cane toads have been invading large coastal areas, at an accelerated speed: the invasion started with a speed of 10 kilometres a year, and continuously accelerated to the impressive speed of 55 kilometres a years today \cite{Urban}. The mechanism for this acceleration is documented \cite{Lindstrom}: the individuals close to the invasion fronts have an anomalously high dispersion rate, and drive the invasion front.

A model introduced in 1937 by Fisher in \cite{Fisher} (and simultaneously in \cite{KPP}) has proven very useful to describe biological invasions \cite{Shigesada,Hastings}. This model describes the dynamics of the density of a population. 
In a homogeneous environment, a population which is initially present on a limited set only will propagate at an asymptotically constant speed \cite{Bramson}, with a certain profile, called travelling wave \cite{KPP}. The study of travelling waves and related propagation phenomena has prompted a large mathematical literature, we refer to \cite{Xin} for a review on this active field of research. Recently, more surprising dynamics have been uncovered: in \cite{Roques}, it has been shown that a slowly decaying initial condition may lead to accelerating invasion fronts. Similar dynamics can be observed for compactly supported initial populations if the diffusion operator of the Fisher-KPP equation is replaced by a nonlocal dispersal operator with fat tails \cite{Kot,Garnier}, or by a fractional diffusion operator \cite{Coulon}. Finally, in \cite{Bouin}, it was proven that a similar dynamics can be observed when the diffusion operator is replaced by a kinetic operator modelling a run and tumble dynamics.

The phenomena that we want to describe here is different from the ones described above: in our case, the acceleration dynamics is due to the continual selection of individual with enhanced dispersion abilities. To model such phenomena, involving both a spatial dynamics of the population and evolutionary phenomena (see \cite{Hairston,Lambrinos}), the population should be structured by a phenotypic trait as well as a spatial variable. Starting from an Individual Based model of such a population, a large population limit can be performed \cite{Fournier} to obtain a non-local parabolic equation. Related models have been studied in e.g. \cite{Prevost,Alfaro}. The case where the phenotypic trait structuring the population is the dispersion rate of the population has been introduced in \cite{Benichou}.

\subsection{The model}
\label{subsec:model}

We will consider a population described by its density $v=v(t,x,\theta)$, where $t\geq 0$ is the time variable, $x\in\mathbb R$ a spatial location, and $\theta\in (1,\infty)$ a phenotypic trait. The dynamics of the population is given by the following model:


Model (NLoc): 
\begin{align*}
  \left\{ 
  \begin{array}{l} \ds
   \partial_t v = \frac \theta 2 \Delta_x v + \frac 1 2\Delta_\theta v + v \left( 1-
    \langle v \rangle \right) \\[3mm] \ds
   v = v(t,x,\theta), \ t \ge 0, \ x \in \R, \ \theta \ge 1, \\[3mm] \ds
   \langle v \rangle(t,x,\theta) := \int_{\max(\theta- A,1)} ^{\theta+A} v(t,x,\omega) \dd
  \omega , \\[4mm] \ds
   v(0,x,\theta) = v_0(x,\theta) \ge 0,\\[3mm] \ds
  \partial_\theta v(t,x,1)=0,\ t \ge 0, \ x \in \R.
  \end{array}
\right.
\end{align*}
In this model, we assume that individuals diffuse through space at a rate given by the phenotypic trait $\theta$. This phenotypic trait $\theta\in[1,\infty)$ is itself submitted to mutations, which appears in the model as a diffusion term in the variable $\theta$, at a rate constant rate $1$ independent from $x$ and $\theta$. We assume that the growth rate of the population in the absence of intra-specific competition is $1$, and is in particular independent of the spatial location $x$ and phenotypic trait $\theta$. We assume that the individuals are in competition with the individuals present in the same location, provided their phenotypic traits are not different, which is quantified by $A>0$. Note that (NLoc) would correspond to the model introduced in \cite{Benichou} if $A=\infty$; we will however always consider here that $A>0$ is finite. We also assume that the individuals reproduce asexualy: during sexual reproductions, recombinations of the DNA strains happen, which leads to very different mathematical models \cite{MR}.

From a modelling point of view, assuming that the phenotypic trait $\theta$ can take arbitrarily large values may appear surprising. It seems however to be a reasonable assumption in this context: an artificial selection experiment \cite{Weber} has shown that it is possible to increase the dispersion rate of flies a hundred folds in just a hundred generations, with little impact on the reproduction rate of the individuals. The field data obtained in \cite{Urban} suggest that the set of possible phenotypic traits does not have a limiting effect on the evolution of dispersal in cane toad populations. The data collected in \cite{Lindstrom} provides some indications on how rapid evolution of the dispersion rate is possible: tracking data of the cane toads show that the animals alternate resting phases and ballistic motion, and the individuals at the front of the invasion simply have longer ballistic phases, and a higher directional persistence. These simple modifications of individual motion has limited energetic cost, while greatly increasing individuals dispersion rate.

\subsection{The main results}
\label{sec:main-results}

We make the following assumptions on the initial condition: 
\begin{enumerate}
\item  (Compact support in $\theta$.) We have that $u_0(x, \theta) = 0$ unless $\theta_{\min} \le \theta \le \theta_{\max}$ for some $\theta_{\max} >\theta_{\min} \ge  1$; 
\item (Thin tail) We have, for some $C,c>0$ $u_0(x, \theta) \le C \exp( - cx )$ uniformly over $x$ and $\theta$, and $\inf_{\mathbb R_-\times [\theta_{\min}', \theta_{\max}']}u_0>0$, for some $\theta_{\max}' >\theta_{\min}' \ge  1$;
\item (Regularity.) We assume that $\left((x,\theta)\mapsto \bar v_0(x,\theta=v_0(x,|\theta|+1)\right)\in C^{3}(\mathbb R^2)$, that is 
\[\sum_{0\leq k+l\leq 3}\|\partial_x^k\partial_\theta^l \bar v_0\|_{L^\infty(\mathbb R^2)}<\infty.\]
\end{enumerate}

We can now state the main result of this study, which describes the acceleration of the invasion front:
 
\begin{theorem} \label{T:toads-nonlocal} 
Let $v_0\in C^{2+\delta}(\mathbb R\times [1,\infty))$ with compact support in $\theta$, thin tail in $x$ and regular, as described in Subsection~\ref{sec:main-results}. Let $v(t,x, \theta)$ denote the corresponding solution of (NLoc). For $x \in \R$, let $S(t,x) = \sup_{\theta} v(t,x, \theta)$ and let $$\gamma_0 = \frac{2}{3}2^{1/4}.$$ 
We have for all $\gamma > \gamma_0$, 
\begin{equation}\label{toads:UB-nonlocal}
 \lim_{t\to\infty} \sup_{x > \gamma t^{3/2}}  S(t, x) \to 0 
\end{equation}
while for $\gamma < \gamma_0$ 
\begin{equation}\label{toads:LB-nonlocal}
\liminf_{t\to\infty} \inf_{x < \gamma t^{3/2} } S(t,x) >0.
\end{equation}
In other words the population spreads in space as $\gamma_0 t^{3/2}$.
\end{theorem}

An example of an initial condition which satisfies the assumptions (1) and (2) is given by $u_0(x, \theta) = H(x)  \indic{\theta \in (1,2)}$, where $H$ is the Heavyside function. This is a good example to keep in mind for this result, and as a matter of fact much of the proof relies on the analysis of this example, for a modified model (where the non local competition is replaced by a local term, see (Loc)). 
Note however that the thin tail condition we make here is much weaker than the condition that is usually made for propagation front problems: for the Fisher-KPP equation $\partial_t n(t,x)-\frac 12\Delta_x n(t,x)=n(t,x)(1-n(t,x))$, the solution propagates at speed $\sqrt 2$ provided the tail of initial condition satisfies $n(0,x)\leq C e^{\sqrt 2 x}$ for some $C>0$. If the tail of the initial condition decreases slower than that, the solution can propagate much faster than that, and indeed, for any $c\geq \sqrt 2$ there exist a travelling wave propagating at speed $c$ (see e.g. \cite{Xin}). Note that if tail of the initial condition of (NLoc) decreases polynomially only, we expect that the population could propagate faster than what we describe here, just as it happens for the Fisher-KPP equation \cite{Roques}.

%
\begin{rmk} \label{Rmk:generalised-fronts}
This description of the propagation of the population is close to the notions of \emph{spreading speed} (see e.g. \cite{Aronson-Weinberger,Hamel}), and generalized travelling wave (see \cite{Berestycki-Hamel}). Indeed, the solution $u=u(t,x):\mathbb R_+\times\mathbb R\to\mathbb R_+$ of the Fisher-KPP equation is said to spreading at speed $2$, in the sense that if $u(0,\cdot)\neq 0$, $u(0,\cdot)\geq 0$ is compactly supported, then for any $c_-<2<c_+$,
\[\lim_{t\to\infty} u(t,c_-t)=1,\quad \lim_{t\to\infty} u(t,c_+t)=0.\]
The description of the solution's dynamics in Theorem~\ref{T:toads-nonlocal} can thus be seen as an extension of this spreading speed.

Note that the estimates \eqref{toads:UB-nonlocal} and \eqref{toads:LB-nonlocal} provide a precise description on the acceleration of the invasion front: it does provide the exponent $t^{\frac 23}$ of the acceleration (see e.g. \cite{Bouin} for a result of this type), but it is indeed much more precise: We provide the exact multiplicative constant $\gamma_0$ in front of the leading term $\gamma_0 t^{\frac 32}$. 
\end{rmk}

Before describing the key ideas of the proof, let us discuss a natural generalizations of Theorem~\ref{T:toads-nonlocal}, where the dispersion rate of the population is not given by the phenotypic trait $\theta$, but by $\theta^\alpha$, for some $\alpha>0$. The equation on $v$ would then become:
\begin{equation}\label{model-alpha}
\partial_t v=\frac{\theta^\alpha}2\Delta_x v+\frac 12\Delta_\theta v+v(1-\langle v\rangle).
\end{equation}
We believe the framework of our analysis could be used to show that 
\[\lim_{t\to\infty}\sup_{x\geq t^{\kappa}}\int v(t,x,\theta)\,d\theta =0,\quad \liminf_{t\to\infty}\sup_{x\geq t^{\kappa'}}\int v(t,x,\theta)\,d\theta >0,\] 
for any $\kappa'<\frac{2+\alpha}2<\kappa$, this analysis is however beyond the scope of this study. Note that the case of (NLoc), leads to an acceletarion $x\sim t^{\frac{2+\alpha}2}=t^{\frac 32}$, that is close to the observation from \cite{Urban} on the invasion of Cane toads in Australia, which justifies our particular focus on this case. 


Another possible generalization of this model is to consider a competition term that is non local in both trait and phenotype. If we assume that the spacial non-locality of this competition is related to individual dispersal, a natural model to consider is
 \[\partial_t w = \frac{\theta}2 \Delta_x w +\frac 12 \Delta_\theta w + w \left( 1-    \langle w \rangle \right),\]
where $\langle w \rangle(t,x,\theta) := \frac{1}{\sqrt \theta} \int_{x-\alpha \sqrt \theta}^{x+\alpha \sqrt \theta}\int_{\min(\theta- A,1)} ^{\theta+A} w(t,y,\omega) \dd\omega \dd y$. The dynamics of this other model can indeed be described with an approach similar to the one presented here: simple a priori estimates show that $\langle w \rangle $ is uniformly bounded. A De Giorgi-Moser iteration scheme (see \cite{Moser}) can then be used to show that $w$ is indeed uniformly bounded. One can then compare the dynamics of this model with the local model (Loc), as done in this study (see Section~\ref{sec:comparison-models}).

Finally, let us mention that in (NLoc), it would be natural to consider the case where $A=\infty$. This is actually the model that was introduced in \cite{Benichou}. In Figure~\ref{fig:4}, numerical simulations show that the description of the dynamics of (NLoc) seem to apply to the case where  $A=\infty$ also. Proving this result would however require additional estimates.

\subsection{Discussion on the dynamics of the solutions}\label{subsec:discussion}

In Theorem~\ref{T:toads-nonlocal}, we show that the position of the invasion front is well approximated by
\begin{equation}\label{eq:dyn-x}
x(t)= \frac {2^{5/4}}3 t^{3/2}.
\end{equation}
Indeed, the proof also provides some information on the phenotypic trait present at this front: at $x(t) = \gamma t^{3/2}$,
$v(t,x(t), \theta(t)) >C>0$ for 
\begin{equation}\label{eq:dyn-theta}
\theta(t) = \frac {\sqrt{2}} 2 t.
\end{equation}
Moreover it would be relatively easy to show that near the edge of the invasion  front, that is in $x\sim \gamma_0 t^{3/2}=({2^{5/4}}/3) t^{3/2}$, \emph{all} particles have a mobility of approximately $\theta (t)= (\sqrt{2}/2)t$.

%
%
%
%
If we considered a linearisation of (NLoc), and a situation where $v$ is independent of $x$, then the solution would propagate towards large $\theta>1$ at speed $\sqrt 2$. It is worth noticing that the mobility $\theta$ found at the edge of the propagation front increases at only half this speed, $\theta = (\sqrt{2} / 2)t$. This dynamics is then indeed the effect of a combination of  evolutionary and spatial dynamics.



\medskip

To validate the quantitative approximations \eqref{eq:dyn-x} and \eqref{eq:dyn-theta}, we performed some numerical simulations of (NLoc), (Loc) and (NLoc) with $A=\infty$. The simulations are based on a finite difference scheme, with some additional Neuman boundary conditions at the edge of the $x$ interval we consider. The numerical results are in good agreement with the theoretical results \eqref{eq:dyn-x} and \eqref{eq:dyn-theta} in each of the three cases: (NLoc) (Figure~\ref{fig:1} which corresponds to Theorem~\ref{T:toads-nonlocal}), (Loc) (Figure~\ref{fig:3} which corresponds to Theorem~\ref{T:toads}), and (NLoc) with $A=\infty$ (Figure~\ref{fig:4} for which we do not have theoretical results).

For (NLoc), which is the main focus of this study, we provide in Figure~\ref{fig:2} a more precise comparison of the numerical position and phenotype at the front with the theoretical approximations \eqref{eq:dyn-x} and \eqref{eq:dyn-theta}. The approximations developed in this study seem to provide a good description of the dynamics of solutions.


\begin{figure}[h]
\centering
\includegraphics[width=70mm]{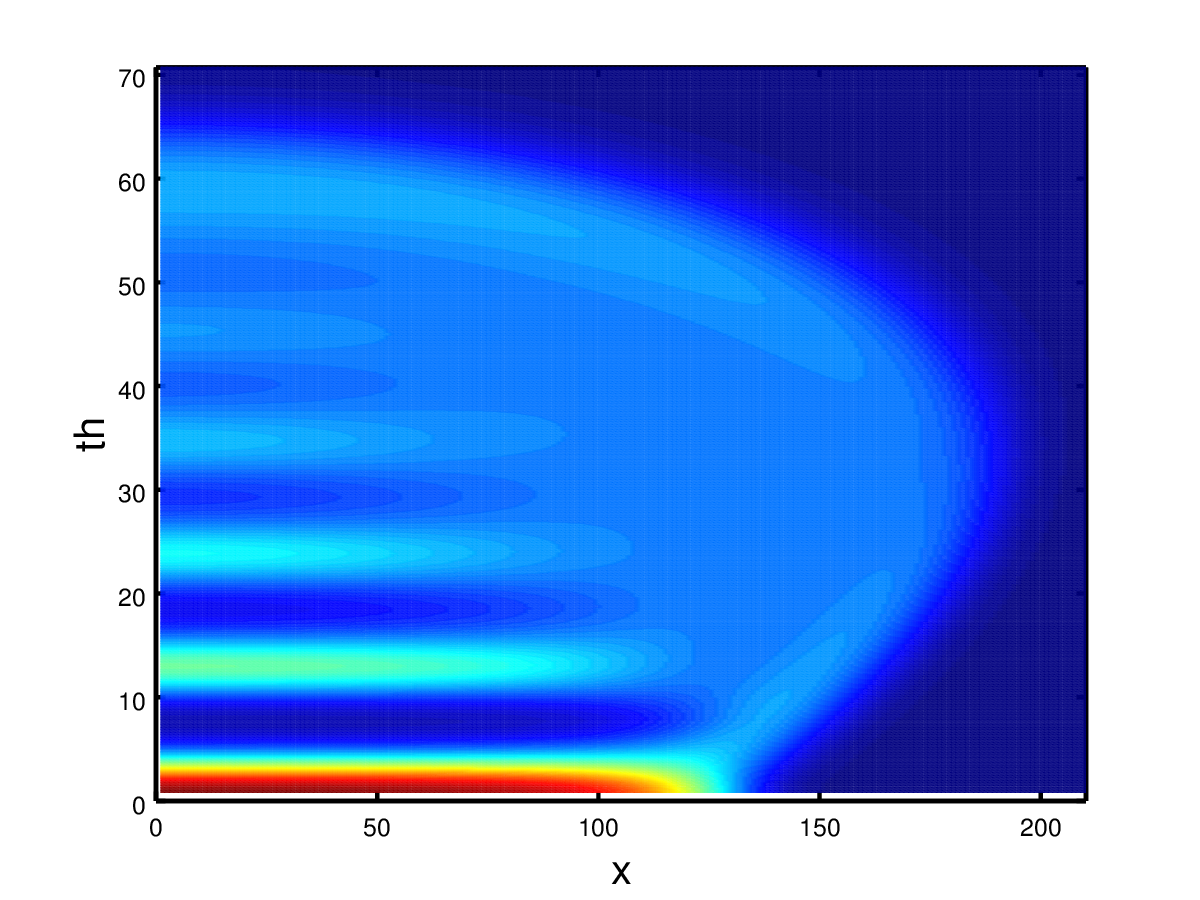}\quad \includegraphics[width=70mm]{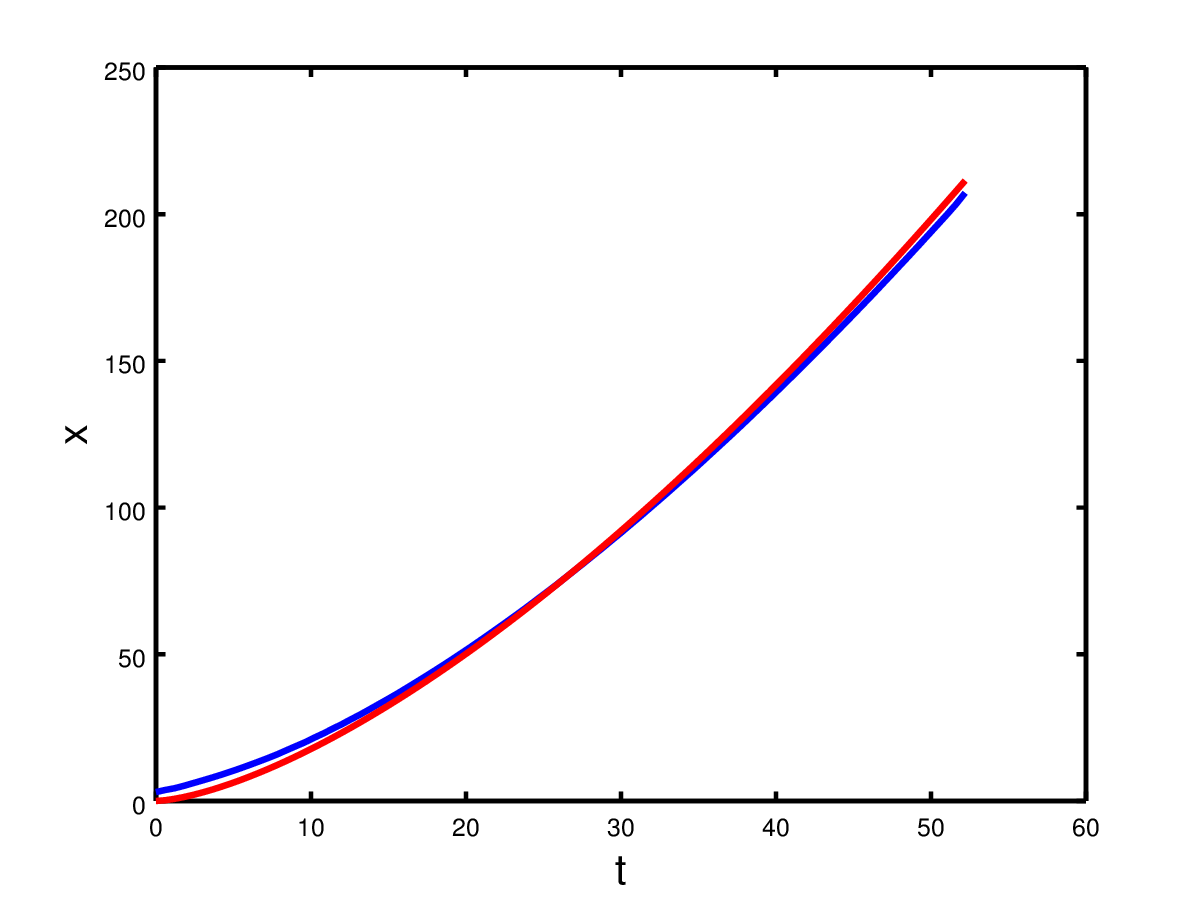}\quad \includegraphics[width=70mm]{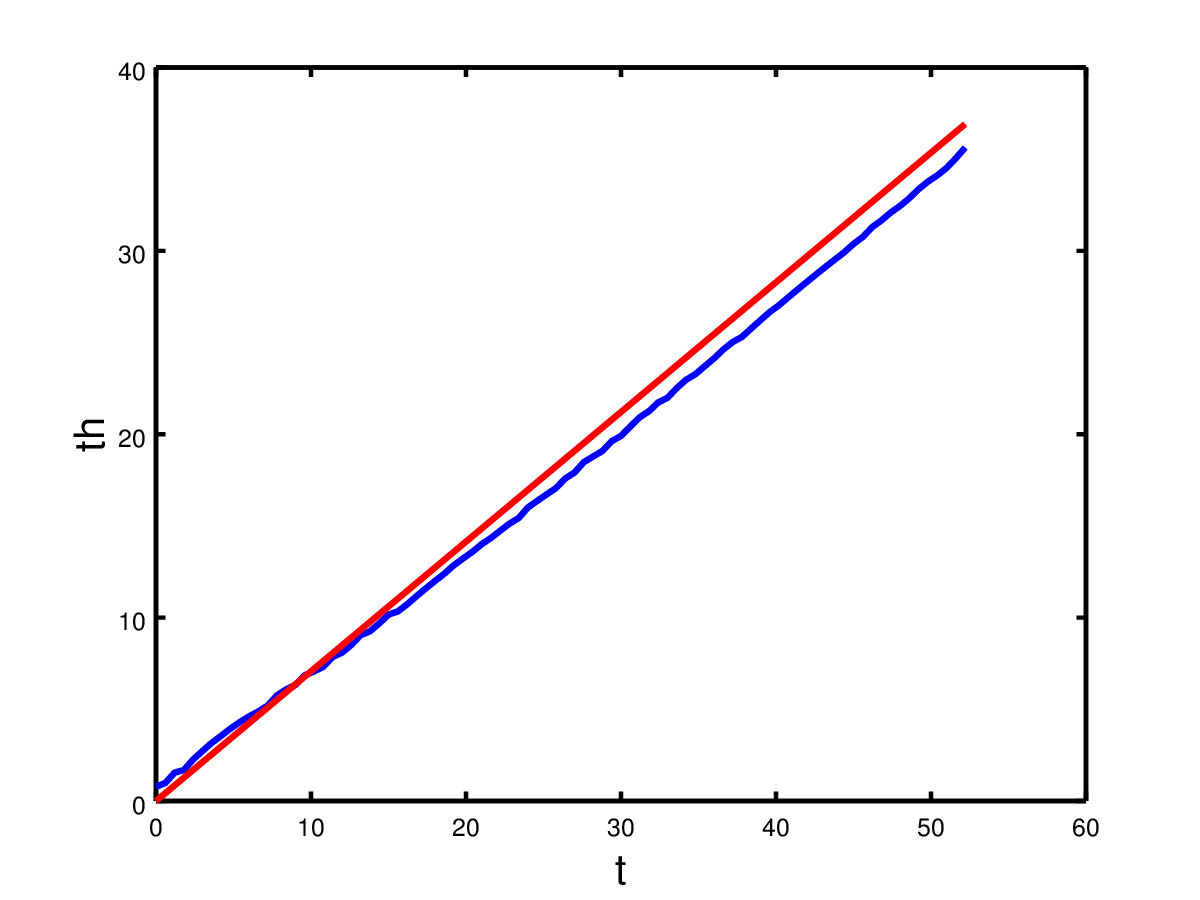}
\caption{Numerical simulation of (NLoc). The first graph represents $v(53,\cdot,\cdot)$. In the second graph, the red curve represents the theoretical position of the front, $x(t):t\mapsto (2^{5/4}/3)t^{3/2}$ (see \eqref{eq:dyn-x}), while the blue curve represent the result of the numerical simulation. Similarly, in the last graph, the red curve represents the theoretical phenotypic trait present at the front, $\theta(t):t\mapsto (\sqrt 2/2)t$ (see \eqref{eq:dyn-theta}), while the blue curve represent the result of the numerical simulation.}
\label{fig:1} 
\end{figure}

\begin{figure}[h]
\centering
\includegraphics[width=70mm]{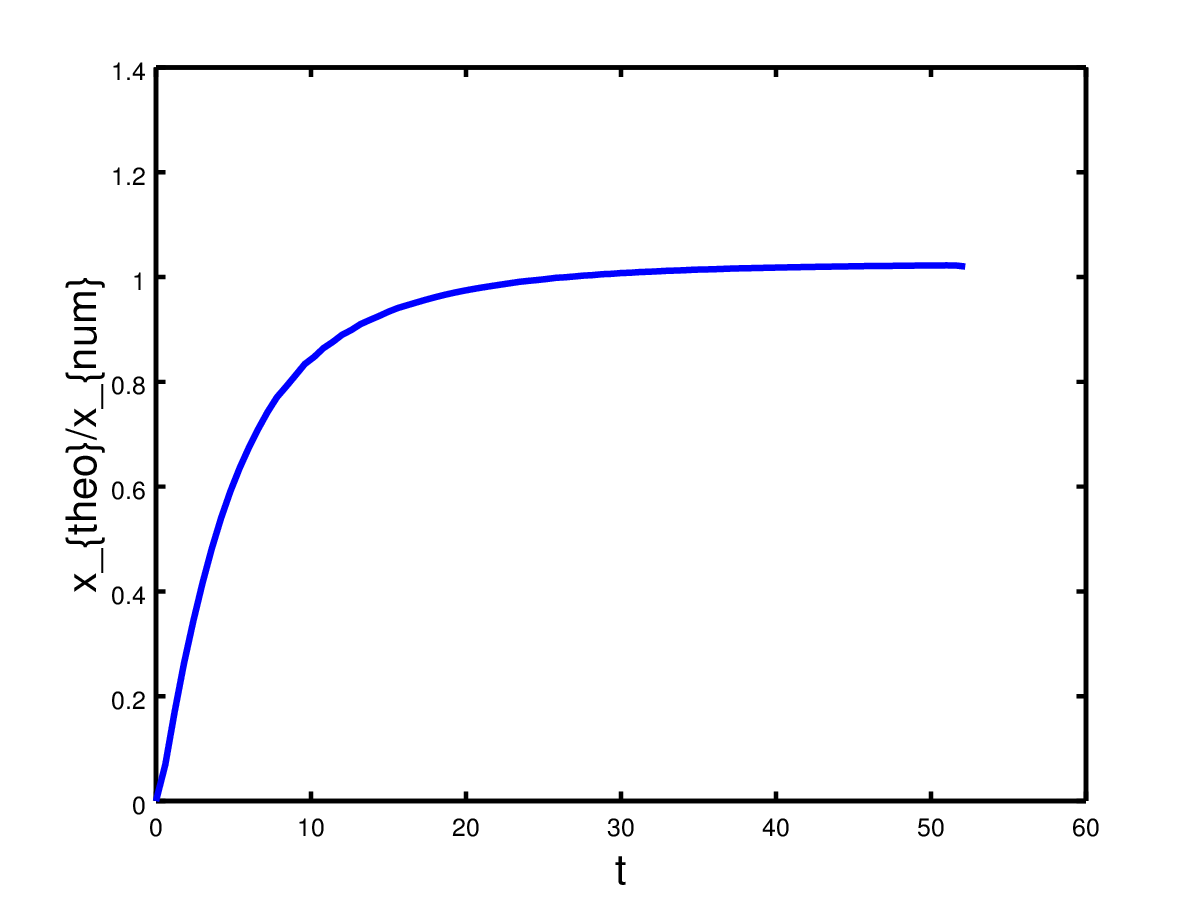}\quad \includegraphics[width=70mm]{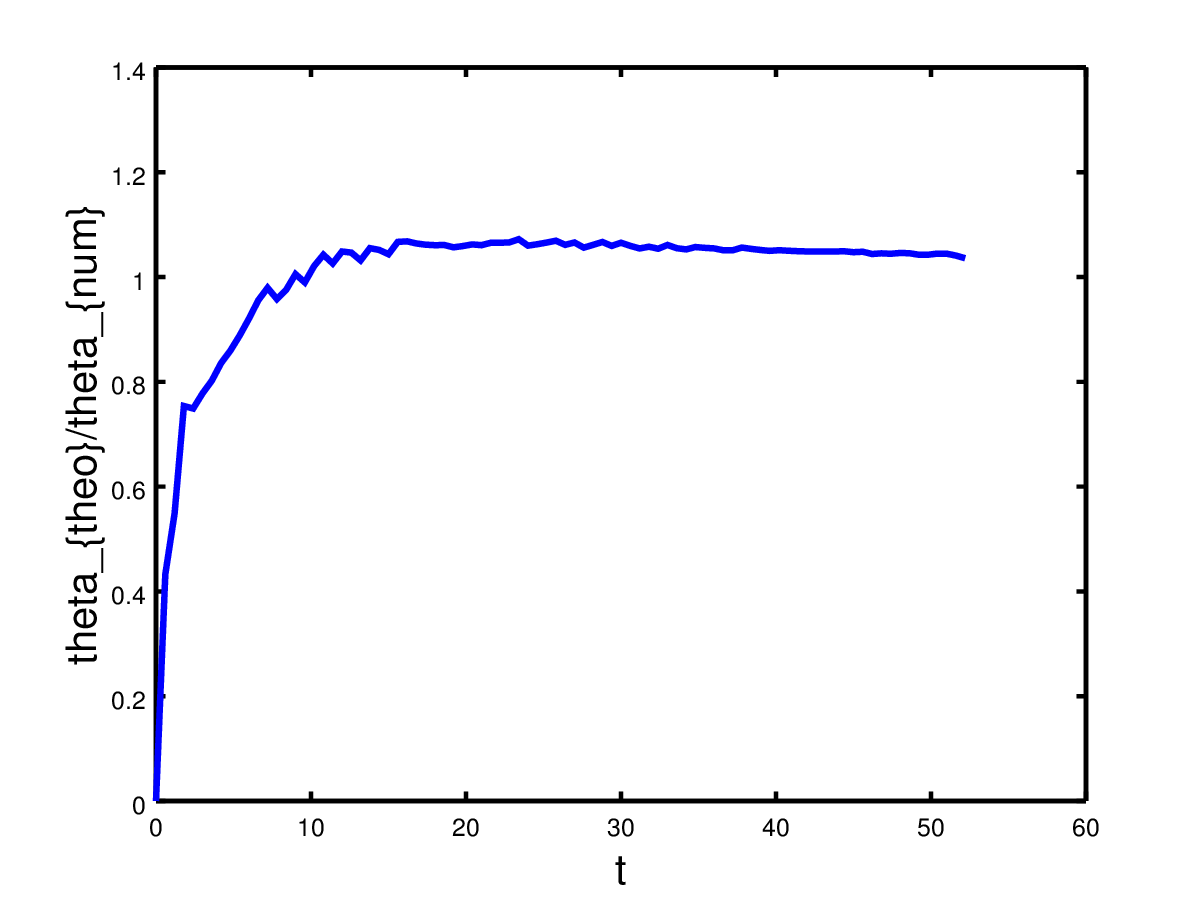}
\caption{Numerical simulation of (NLoc). The first graph represents the quotient of the theoretical position  $x(t):t\mapsto (2^{5/4}/3)t^{3/2}$ of the front (see \eqref{eq:dyn-x}) divided by the corresponding quantity obtained by simulation. Similarly, the second graph represents the quotient of the theoretical phenotypic trait  $\theta(t):t\mapsto (\sqrt 2/2)t$ present at the front (see \eqref{eq:dyn-theta}) divided by the corresponding quantity obtained by simulation.}
\label{fig:2} 
\end{figure}


\begin{figure}[h]
\centering
\includegraphics[width=70mm]{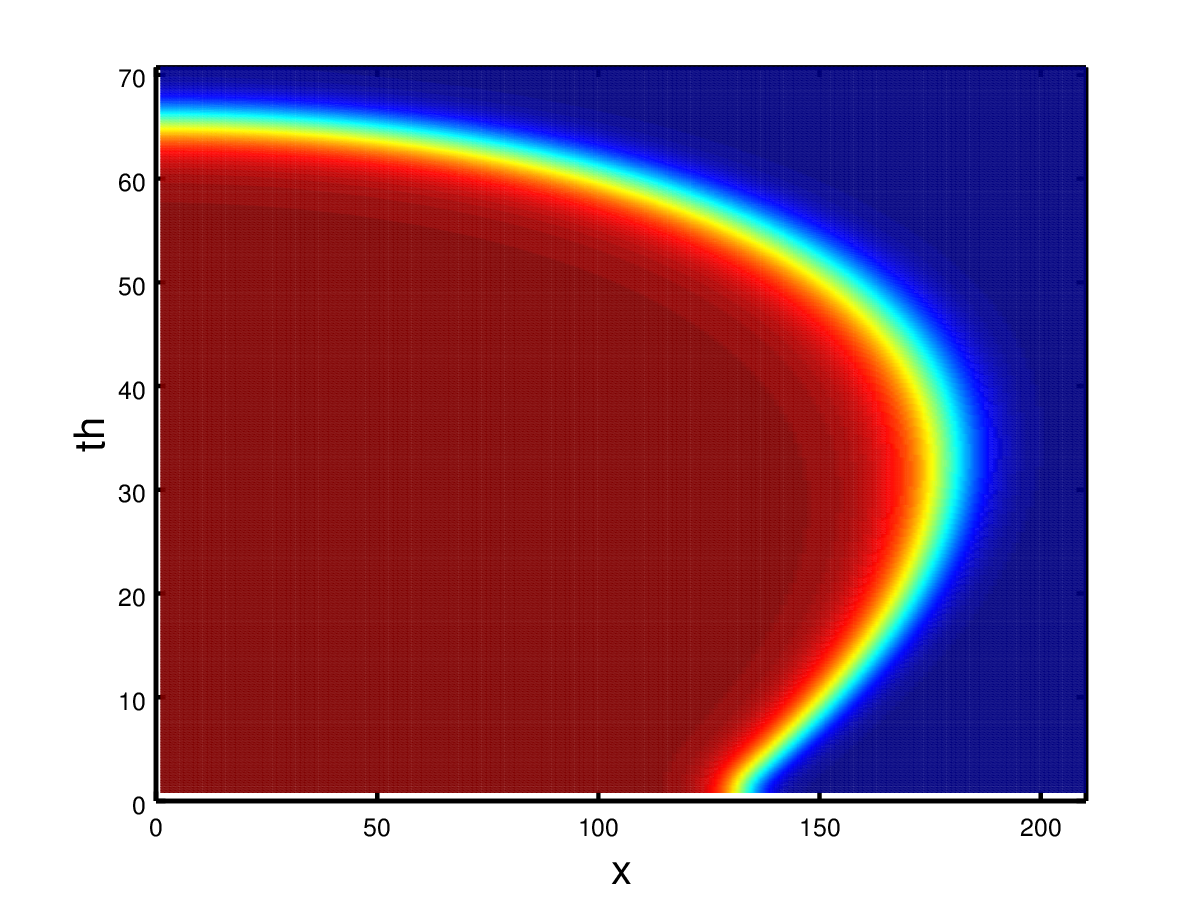}\quad \includegraphics[width=70mm]{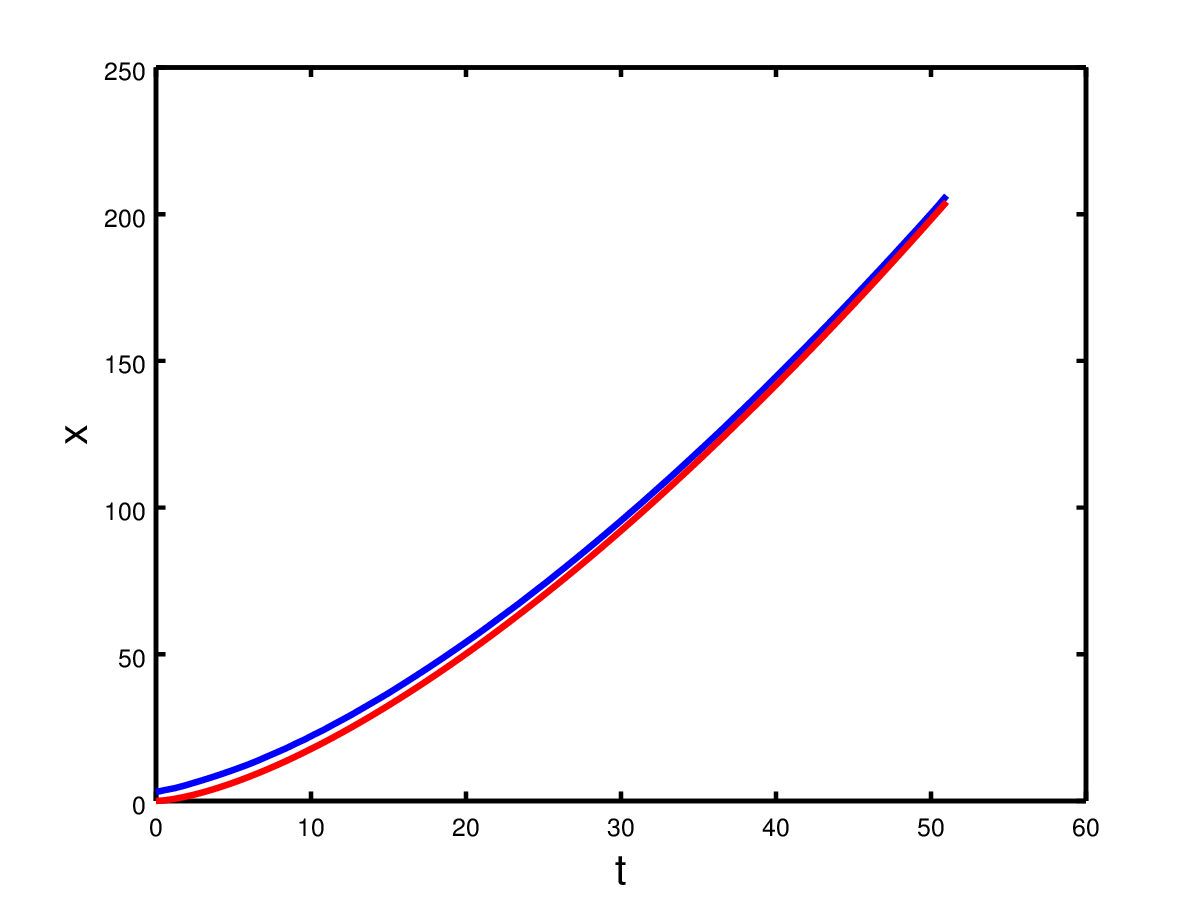}\quad \includegraphics[width=70mm]{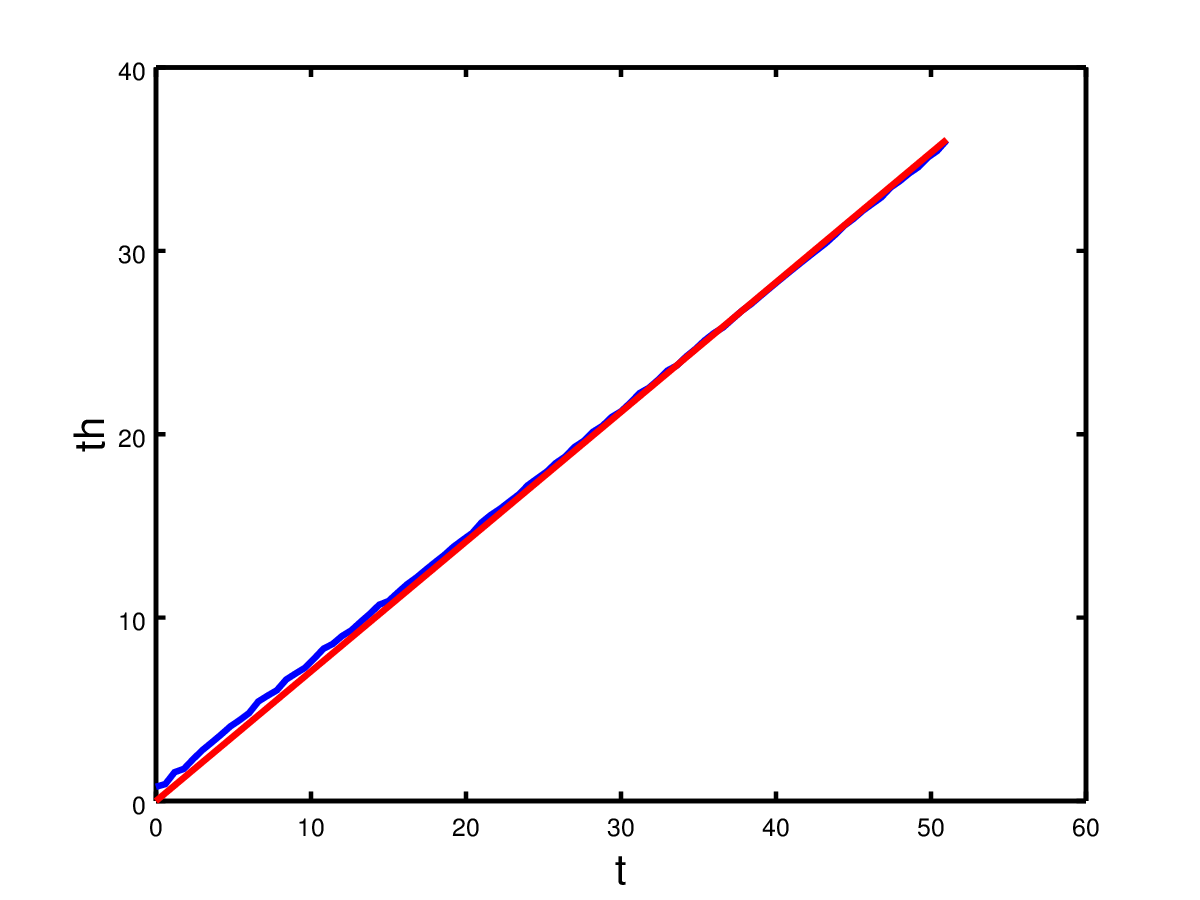}
\caption{Numerical simulation of (Loc). The first graph represents $v(53,\cdot,\cdot)$. In the second graph, the red curve represents the theoretical position  of the front, $x(t):t\mapsto (2^{5/4}/3)t^{3/2}$ (see \eqref{eq:dyn-x}), while the blue curve represent the result of the numerical simulation. Similarly, in the last graph, the red curve represents the theoretical phenotypic trait present at the front, $\theta(t):t\mapsto (\sqrt 2/2)t$ (see \eqref{eq:dyn-theta}), while the blue curve represent the result of the numerical simulation.}
\label{fig:3} 
\end{figure}

\begin{figure}[h]
\centering
\includegraphics[width=70mm]{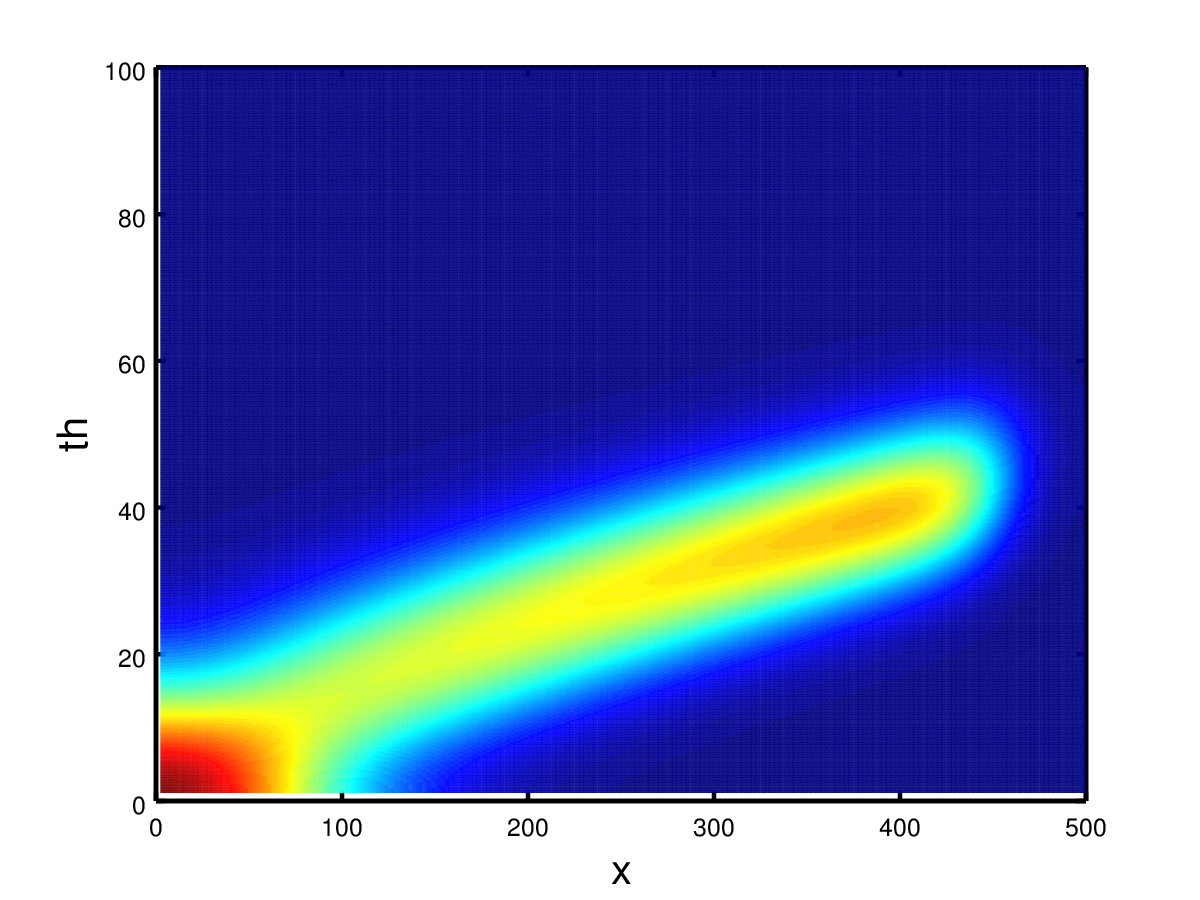}\quad \includegraphics[width=70mm]{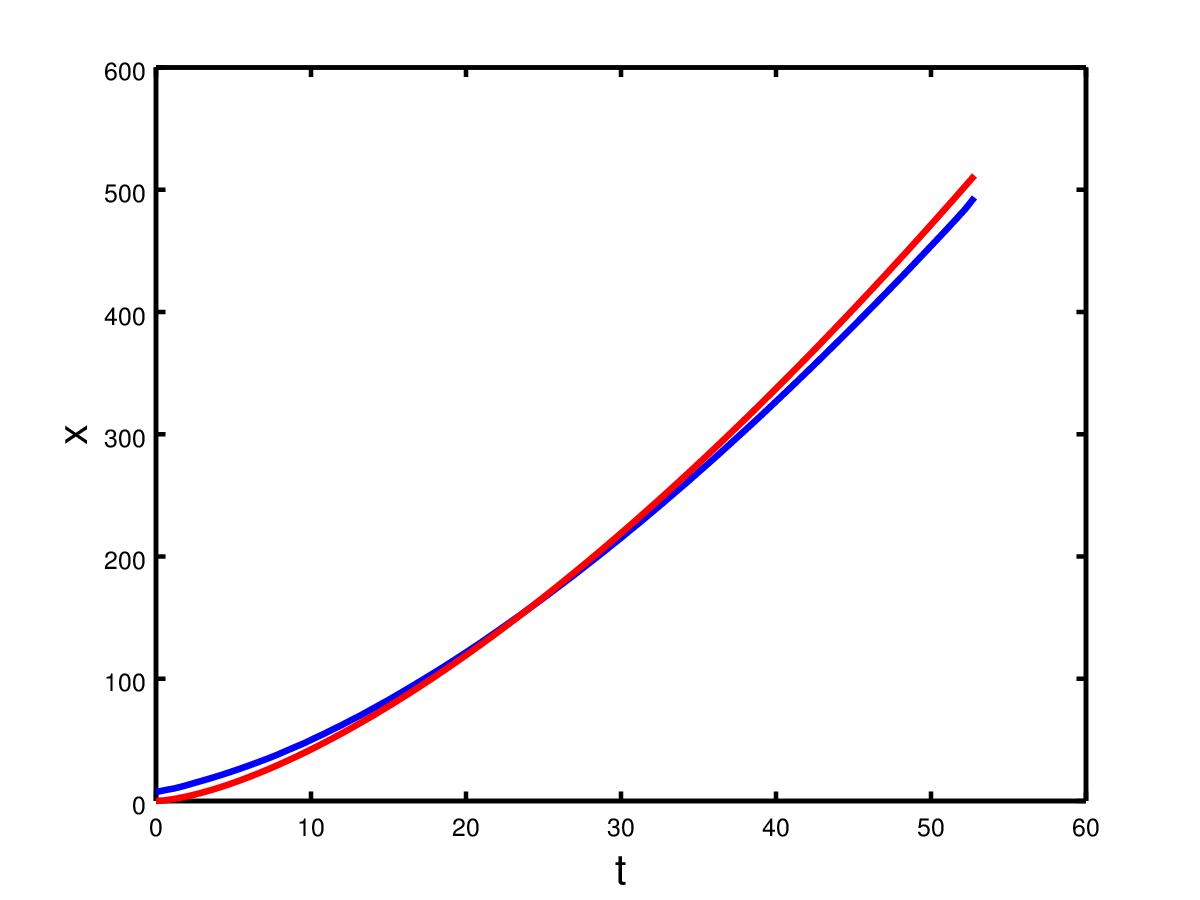}\quad \includegraphics[width=70mm]{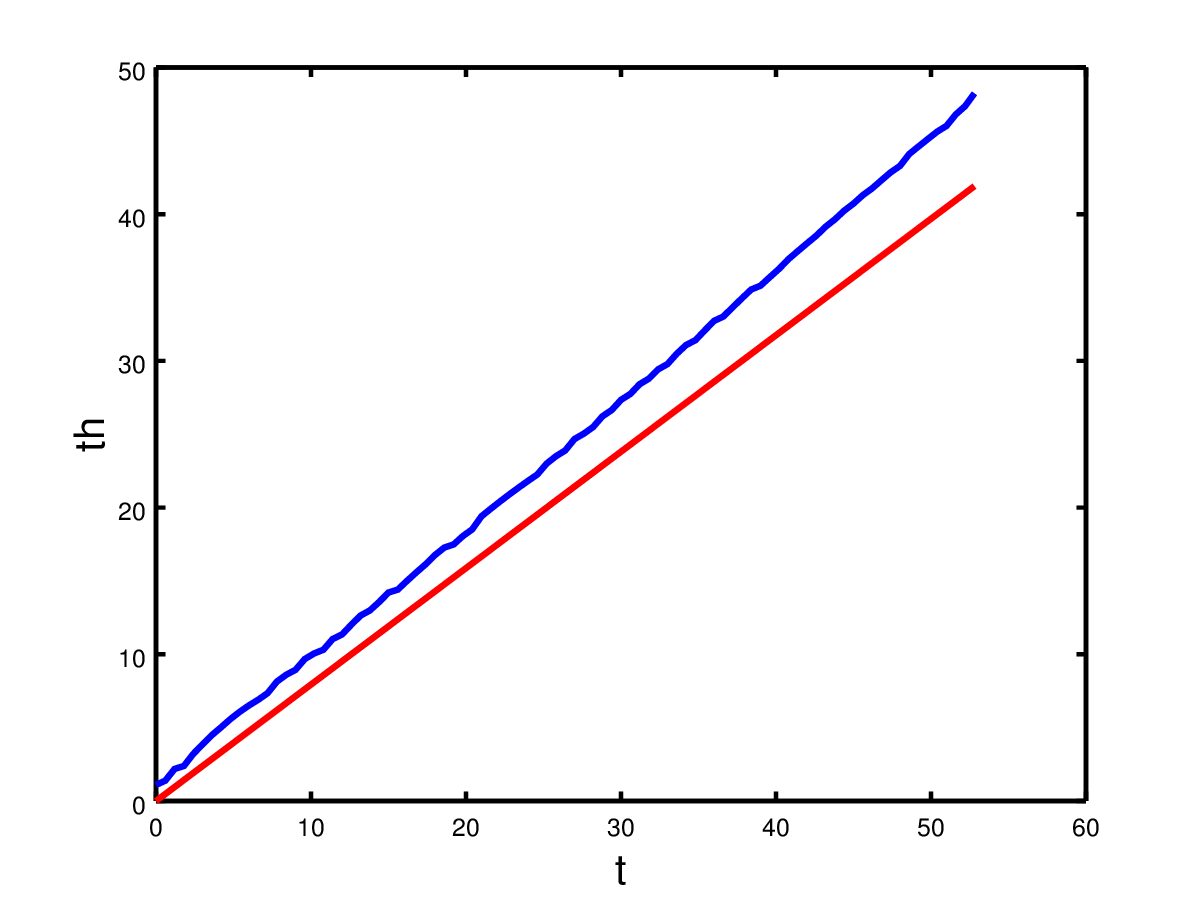}
\caption{Numerical simulation of (NLoc) with $A=\infty$. The first graph represents $v(53,\cdot,\cdot)$. In the second graph, the red curve represents the theoretical position of the front, $x(t):t\mapsto (2^{5/4}/3)t^{3/2}$ (see \eqref{eq:dyn-x}), while the blue curve represent the result of the numerical simulation. Similarly, in the last graph, the red curve represents the theoretical phenotypic trait present at the front, $\theta(t):t\mapsto (\sqrt 2/2)t$ (see \eqref{eq:dyn-theta}), while the blue curve represent the result of the numerical simulation.}
\label{fig:4} 
\end{figure}

\subsection{Key ideas of the proofs}
\label{sec:key-ideas-proofs}

The first difficulty in the analysis of the model (NLoc) is to derive a uniform $L^\infty$ bound on the solution. This difficulty already appeared in \cite{Bouin}, where an $L^\infty$ bound was constructed for travelling waves of (NLoc) (that is steady-states of (NLoc), with an additional drift term), provided the set of phenotypic traits is bounded: $\theta\in[1,\bar \theta]$. A bound of this type was also derived for a parabolic model in \cite{Turanova}, still for a bounded set of phenotypic traits. In Section~\ref{sec:comp-(II)}, we will prove that an $L^\infty$ bound on the solution $v$ of (Nloc) can be established, even when the set of phenotypes is unbounded, that is $\theta\in(1,\infty)$. The proof will be based on a generalization, and a simplification of the argument of \cite{Turanova}.

Equipped with this uniform in time estimate, we will be able to show, in Section~\ref{sec:comparison-models} that the dynamics of solutions of (NLoc) is similar to the dynamics of the following parabolic model, where the non-local competition kernel $-\langle v\rangle$ is replaced by a local competition $-u$:

Model (Loc): 
\begin{align*}
  \left\{ 
  \begin{array}{l} \ds
   \partial_t u = \frac\theta 2 \Delta_x u + \frac 12\Delta_\theta u + u \left( 1-
    u \right) \\[3mm] \ds
   u = u(t,x,\theta), \ t \ge 0, \ x \in \R, \ \theta \ge 1, \\[3mm] \ds
   u(0,x,\theta) = u_0(x,\theta) \ge 0,\\[3mm] \ds
  \partial_\theta u(t,x,1)=0,\ t \ge 0, \ x \in \R.
  \end{array}
\right.
\end{align*}

To description of the dynamics of the solutions of this parabolic equation, we will use a probabilist representation of those solutions through branching brownian motions. This idea was introduced by McKean in \cite{mckean}, who showed that the solution of the Fisher-KPP equation, that is
\[\partial_t u-\frac 12\Delta_x u=u(1-u),\]
with the initial condition $u(0,x):=1_{x\leq 0}$ is given by
\[u(t,x)=1-\E^x\left[\prod_{i\in I(t)}1_{X_i\leq 0}\right],\]
where $I(t)$ is the number of vertex at time $t$ in a branching Brownian tree, stated at time $0$ at location $x\in\mathbb R$, and $X_i(t)$ is the location of the vertex $i$ at time $t$. Notice that this stochastic representation of the solutions of the Fisher-KPP equation is very different from the Individual-Based Model underlying this equation.

We will show in Section~\ref{sec:preliminaries-proba} that this representation of McKean can be extended to the solutions of (Loc), using a branching Brownian tree in $2D$. This Branching Brownian Motion is not standard: to represent the fact that the dispersion in $x$ is given by $\theta$ in (NLoc), the dispersion of the particles of the BBM in the two directions will be coupled. The result of this analysis for (Loc) can be summed up in Theorem~\ref{T:toads}, where we will use a slightly modified assumption on the initial condition:
\begin{enumerate}
\item[(2')] (Thin tail) We have, for some $C,c>0$ $u_0(x, \theta) \le C \exp( - cx )$ uniformly over $x$ and $\theta$, and $u_0(x,\theta)\to 1$ as $x \to - \infty$, uniformly over $\theta \in [\theta_{\min}, \theta_{\max}]$.
\end{enumerate}

\begin{theorem} \label{T:toads} 

Let $u_0\in L^\infty(\mathbb R\times\mathbb R_+)$ with compact support in $\theta$ as described in Subsection~\ref{sec:main-results}, and a thin tail (see (2') above). Let $u(t,x, \theta)$ denote the corresponding solution of (Loc) with either Neuman or Dirichlet boundary condition on $\{\theta=0\}$. For $x \in \R$, let $S(t,x) = \sup_{\theta} u(t,x, \theta)$ and let $$\gamma_0 = \frac{2}{3}(2)^{1/4}.$$ 
We have for all $\gamma > \gamma_0$, 
\begin{equation}\label{toads:UB}
 \sup_{x > \gamma t^{3/2}}  S(t, x) \to 0 
\end{equation}
while for $\gamma < \gamma_0$ 
\begin{equation}\label{toads:LB}
\inf_{x < \gamma t^{3/2} } S(t,x) \to 1
\end{equation}
as $t \to \infty$.
\end{theorem}
Proving this result will be the core of our study. We will first provide in Section~\ref{sec:upper-bound} an upper bound on the propagation of solutions of (Loc). We will then describe some \emph{optimal} trajectories of the BBM in Section~\ref{sec:trajectories}, which will in then allow us to derive a precise lower bound on the propagation of solutions in Section~\ref{sec:lower-bound}. In Section~\ref{sec:conclusion-proof}, we combine those estimates to conclude the proof of Theorem~\ref{T:toads}.

\section{Probabilistic preliminaries}\label{sec:preliminaries-proba}


We will assume without loss of generality that
$\theta_{\min} = 0, \theta_{\max} = 1$.

\subsection{McKean representation}

To start the proof we will use a McKean representation for the
equation (Loc). To this end we recall the general idea of this
representation.  Let $L$ be the generator of some continuous Markov
process $X$ taking values in $\R^d$. (Thus if $L = (1/2) \Delta$, $X$
is nothing but ordinary Brownian motion).  Let $f_0$ be an initial
measurable data with $0 \le f_0 \le 1$.  Let
$(X_t^i, i \in I_t, t \ge 0)$ be a system of branching diffusions
based on $L$: that is, each particle branches at rate 1, and move
according to the diffusion specified by $L$.  All motions and
branching events are independent of one another, and note that no
particle ever dies. In these notations, $I_t$ is the set of indices of
particles alive at time $t$ (note that $I_t$ is thus never empty). We
label the positions of the particles at some time $t\ge 0$ by
$(X_t^i)_{i \in I_t}$.  Let $\P^x$ denote the law of this system when
there is initially one particle at $x \in \R^d$.
 
\begin{prop} [McKean representation \cite{mckean}] \label{P:mckean0}
Let
$$
u(t,x) =  1 - \E^{x}  \left[ \prod_{i \in I(t)} (1 - f_0(X_t^i ) )\right]
$$
solves the problem:
$$
\begin{cases} 
\frac{\partial u}{\partial t} (t,x) &= Lu + u(1-u)\\
u(0,x) & = f_0(x)
\end{cases}
$$
\end{prop}

 In fact, McKean's result is stated for Brownian motion but it is straightforward to extend the result to a general diffusion. Applying the above result to our setting, we are led to the following representation. Introduce a branching Brownian motion in $\R^2$, where a general element for the first and second coordinates respecitvely will be labelled $x$ and $\theta$. 
 We denote by $I_t$ the set of indices of particles alive at time $t$, and let $N_t  = |I_t|$. We label the positions of the particles by $(W_t^i, \theta_t^i)_{i \in I_t}$. 
In the case of \emph{Dirichlet} boundary conditions, we further kill the particle if it ever touches zero: that is, we consider $\tilde I(t) = \{ i \in I(t) : \inf_{s\le t} \theta^i_{s,t} >0 \}$.
For a fixed $t\ge 0$ and $i \in I_t$, let $(W^i_{s,t}, \theta_{s,t}^i )$ denote the position of the ancestor of the particle $i \in I_t$ at time $s \le t$. We use this to build a new process $(X_t^i, i \in \tilde I_t)_{t\ge 0}$ as follows: we set
\begin{equation}\label{X}
X_t^i =  \int_0^t \sqrt{ \theta_{s,t}^i  } dW_{s,t}^i
\end{equation}
where the integral above is It\^o's stochastic integral with respect to the Brownian motion $(W^i_{s,t}, 0 \le s \le t)$.

\begin{prop} \label{P:mckean} Let $u(t,x, \theta)$ be as in Theorem \ref{T:toads}. 
We have 
\begin{equation}
u(t,x, \theta) = 1- \E^{(x, \theta)} \left[ \prod_{i \in \tilde I_t} (1 - u_0 (X_t^i, \theta_t^i )) \right].\label{mckeanIC}
\end{equation}
In particular, if $u_0(x, \theta) = H(x)  \indic{\theta \in (0,1)}$, where $H$ is the Heavyside function, then 
$$
u(t,x, \theta) = \P^{(0, \theta)} (\exists {i \in I_t}: X_t^i > x \text{ and } \theta_t^i <1 )
$$ 
\end{prop}

\begin{proof}
The only thing which needs to be noted that if $ \theta, W$ are two independent Brownian motions, and if $X_t = \int_0^{t\wedge \tau} \sqrt{\theta_s}dW_s$, then $\{(X_t, \theta_t), t \ge 0\}$, where $\tau$ is the first hitting of zero by $\theta$, forms a Markov process with generator 
$$
L u (x, \theta) = \frac{|\theta|}{2} \frac{\partial^2 u}{\partial x^2} + \frac12 \frac{\partial^2 u}{\partial \theta^2}
$$
and Dirichlet boundary conditions. 
\end{proof}
Note that by convention, when $\tilde I_t$ is empty, the product $\prod_{i\in  \tilde I_t}$ is set equal to 1. 

In the case of Neumann boundary conditions, the McKean representation is similar, except that we set $X_t^i =  \int_0^t \sqrt{ |\theta_{s,t}^i | } dW_{s,t}^i$ for all $i \in I_t$ and the product is over all $i \in I_t$ in the formula \eqref{mckeanIC}.

\subsection{Many-to-one lemma}

We will use repeatedly the so-called many-to-one lemma, which is a trivial but useful way of relating expected sum of functions of particle trajectories in a branching Brownian motion to the expected value of the same function applied to a single Brownian trajectory. 

\begin{lemma}
\label{L:many_to_one}
Let $T$ be a random stopping time of the filtration
$ \cF_t = \sigma(\bar W_i(s), \theta_i(s), i \in I(s), s\leq t)$, and
assume that $T$ is almost surely finite. For any bounded measurable
functional $g$ on the path space $C([0, \infty)^2)$,
\[
 \E \left[ \sum_{i \in I_T} g(  (W_{s,T}^i, \theta_{s,T}^i)_{s \leq T}) \right] = \E [e^T g((W_s, \theta_s))_{s \leq T})],
\]
where $(W_s, \theta_s)_{s \geq 0}$ is a standard planar Brownian motion. 
\end{lemma}

\section{Proof of upper bound for the local equation}\label{sec:upper-bound}


From now on, and until almost the very end of the proof, we will
assume that $n_0(x, \theta) = H(x) \indic{\theta \in (0,1)}$.  With
the use of Proposition \ref{P:mckean} the upper bound \eqref{toads:UB}
from Theorem \ref{T:toads} is easy to prove.

We fix $\theta$ some arbitrary initial value. We fix
$\gamma > \gamma_0$ and call $Z_t = |\mathcal{Z}_t|$, where
$$
\mathcal{Z}_t = \{ i \in \tilde I_t:  X_t^i \ge  \gamma t^{3/2} \text{ and }  \theta_t^i <1\}.
$$ 
From Proposition \ref{P:mckean} we get that
$$
u(t,x_t(1+\eps), \theta) = \P^{(0, \theta)} ( \mathcal{Z}_t \neq \emptyset ) \le \E^{(0, \theta)} (Z_t). 
$$
Hence it suffices to show that this expectation tends to 0 as $t \to \infty$.

We will principally focus on the case of Dirichlet boundary conditions for readability, and make brief comments along the way on how to adapt the arguments to the case of Neumann boundary conditions. The idea will be to consider the particles such that $\int_0^t \theta_s ds$ has a fixed order of magnitude, namely $a t^2$ (and satisfy $|\theta_t| <1$). More precisely we introduce, for a fixed $h>0$, the set of particles $\tilde I_t(a)$ such that
$$
\tilde I_t(a) = \left\{ i \in \tilde I_t: a t^2  \le \int_0^t  \theta_{s,t}^i ds  \le (a+h) t^2 ;  \theta_t^i  < 1\right\}
$$
In the case of Neumann boundary conditions, it is instead the set
$
I_t(a) = \{ i \in I_t: a t^2  \le \int_0^t | \theta_{s,t}^i |ds  \le (a+h) t^2 ; |\theta_t^i|  <1 \}$
which is of interest. 

We first have the following lemma:

\begin{lemma}
\label{L:velocity_theta}
Suppose $a>0$. We have 
\begin{equation}\label{tildeIa}
\E^{\theta_0} [|\tilde I_t(a) | ]  \le \frac{c}{t^{1/2}} e^{t(1- 3a^2/2)}
\end{equation}
for some constant $c$, uniformly over $\theta_0>0$. In the case of Neumann boundary conditions, the same estimate holds with $\tilde I_t(a)$ replaced by $I_t(a)$. 
\end{lemma}
\begin{proof}
We start by $\tilde I_t(a)$. By the many to one lemma, we immediately get
\begin{align*}
\E^{\theta_0}(| \tilde I_t(a)|  ) & = e^t \P^{\theta_0}\left( a t^2  \le \int_0^t \theta_s ds   \le (a+h) t^2  ; \inf_{s\le t}  \theta_s > 0 ; \theta_t<1 \right)\\
&\le e^t \P^{\theta_0}\left( a t^2  \le \int_0^t \theta_s ds   \le (a+h) t^2 ; \theta_t<1  \right)
\end{align*}
We will see that  $$\P^{\theta_0}\left(  \int_0^t \theta_s ds \ge at^2   ; \theta_t <1 \right)\sim  \frac{c}{t^{1/2}} e^{t(1- 3a^2/2)}$$ and \eqref{tildeIa} will then follow.

Note that, by time-reversibility, under $\P^{\theta_0}$, $(\tilde \theta_s: = \theta_t - \theta_{t - s}, 0 \le s \le t)$ is a Brownian motion started from 0 and run for time $t$. Furthermore, if $\theta$ did not hit zero, and $\theta_t \le 1$, and $ \int_0^t \theta_s ds \ge at^2 $, then we certainly have that $ at^2 - t \le \int_0^t \tilde \theta_s ds$. 
Therefore, 
$$
\P^{\theta_0}\left( a t^2  \le \int_0^t \theta_s ds   ; \theta_t<1  \right) \le \P^0 \left( at^2 - t \le \int_0^t \theta_s ds \right).
$$

Observe that under $\P^0$, $(\theta_s,s\le t)$ is a centred Gaussian process with covariance $\E(\theta_s \theta_t ) = s\wedge t$. We immediately get that $ \int_0^t \theta_s ds$ is a centred Gaussian random variable with variance $\sigma^2$ which can be explicitly computed:
\begin{equation}\label{variance_int}
\sigma^2 = \int_0^t \int_0^t  (s\wedge u ) ds du = \frac{t^3}{3}. 
\end{equation}
Consequently, using
the easily established fact about standard normal random variables that as $x \to \infty$,
\begin{equation}\label{tail}
\P(X \ge x) \sim x^{-1} \frac{e^{-x^2/2}}{\sqrt{2\pi}} 
\end{equation}
we obtain
\begin{align*}
\P^0\left(  \int_0^t \theta_s ds \ge at^2 - t  \right) &= \P( X \ge (at^2 - t)/\sigma) \\
& \sim \frac{c}{t^{1/2}} \exp \left( - \frac{a^2 t^4}{2 t^3/3} \right) \\
& = \frac{c}{t^{1/2}} \exp \left( - \frac{3 a^2 t}{2} \right)
\end{align*}
as desired. This proves \eqref{tildeIa}. 
For the case of Neumann boundary conditions, the proof is essentially similar, except that in the reversibility argument, we observe that if $\int_0^t | \theta_s | ds \ge at^2$ we also have 
$\int_0^t |\tilde  \theta_s | ds \ge at^2 - t$ if we also know that $|\theta_t | \le 1$. The rest of the proof is similar, with the estimate for the tail $\int_0^1 | \theta_s| ds$ coming from a theorem of Tolmatz \cite{tolmatz} (see also, for a related question, \cite{svante}). 
\end{proof}

From now on we focus on the case of Dirichlet boundary conditions, but the case of Neumann boundary conditions can be treated in exactly the same way thanks to Lemma \ref{L:velocity_theta}.  
\begin{cor}
Let $a_1 = \sqrt{2/3}$ and fix $a >a_1$. Let $B_1$ be the event that there is an $i\le N_t$ such that
\begin{equation}
\label{apriori}
\int_0^t \theta_{s,t}^i ds \ge at^2 \text{ and } |\theta_t^i| <1
\end{equation}
Then $\sup_{\theta_0}\P^{(0, \theta_0) }(B_1) \to 0$. 
\end{cor}

\begin{proof}
Apply the many-to-one Lemma and Lemma \ref{L:velocity_theta} to find that the expected number of particles satisfying \eqref{apriori} tends to 0, and then use Markov's inequality.
\end{proof}

Consider the filtration $\cG_t = \sigma( \theta^i_{s,t}, i \in I_t, s\le t)$. When we condition on $\cG_s$, $X_t^i$ is a Gaussian random variable with variance $\int_0^t \theta_{s,t}^i ds$. Fix $a = 2a_1$.
We get, applying the many-to-one Lemma again, on the event $B_1^c$, 
\begin{align*}
\E^{(0, \theta_0)}(Z_t | \cG_t) & \le \sum_{i \in I_t} \frac{\sqrt{\int_0^t \theta^i_{s,t} ds } }{x} \exp \left( - \frac{x^2}{2 \int_0^t \theta_{s,t}^i ds } \right)  \\
& \le \sum_{n=0}^{2a_1/h} |I_t( n h) |    \frac{\sqrt{nh t^2}}{ct^{3/2}} \exp \left( - \frac{(\gamma t^{3/2})^2}{2  nh t^2  } \right) \\
& = \sum_{n=0}^{2a_1/h} |I_t( n h) |  t^{-1/2}\sqrt{nh}  \exp \left( - \frac{\gamma^2}{2nh} t \right) .
\end{align*} 
Taking expectations, we get
\begin{align}
\E^{(0, \theta_0)}(Z_t) & \le \sum_{n=0}^{2a_1/h} t^{-1/2}\sqrt{2a_1} t^{-1/2} \exp( t(1- 3(nh)^2/2)) \exp \left( - \frac{\gamma^2}{2nh} t \right) \nonumber \\
& \le \frac{(2a_1)^{3/2} }{ht}
 \exp ( t \max_{a\in[0,2a_1]} \left\{1- \frac{3a^2}{2}  - \frac{\gamma^2}{2a} \right\} )  \label{ubExp}
\end{align}
We claim that for $\gamma> \gamma_0$, this maximum, call it $M(\gamma)$, is negative. 
Indeed, let 
$$
\phi(a) = 1- \frac{3a^2}{2}  - \frac{\gamma^2}{2a}.
$$
At a maximum point of $\phi$ we also have $\phi'(a) = 0$ or $-3a + \gamma^2/ (2a^2) = 0$ or 
$$
\gamma^2 = 6a^3,
$$
in which the maximum is given by 
$$
M(\gamma) = 1- \frac{3a^2}{2} - 3 a^2   = 1- \frac{9a^2}{2}. 
$$
But since $\gamma>\gamma_0 = (2/3) 2^{1/4}$ we get that $a > a_0 = \sqrt{2}/3$. It follows that $M(\gamma) < M(\gamma_0) = 0$. 
We deduce that 
\begin{equation}
\P^{(0, \theta_0)}(Z_t \ge 1) \le \E^{(0, \theta_0)}(Z_t \mathbf{1}_{B_1^c})  + \P^{(0, \theta_0)}(B_1) \to 0,
\label{conclUB}
\end{equation} 
uniformly over $\theta_0$, as desired. 

\section{Identifying the relevant stochastic trajectories}\label{sec:trajectories}


In this short section, we discuss the intuitive idea on which the rest
of the proof is based. For a lower bound, the idea is to apply a
second moment argument to say that there are indeed particles at
distance $x_t = \gamma t^{3/2} $ with high probability for
$\gamma < \gamma_0$. In order to do so, we need to identify the
relevant trajectories which make this event possible.

It can be guessed that a particle satisfying $\theta_0 = 0$ and
$\int_0^t \theta_s ds \ge at^2$ fluctuates around a deterministic
function $\bar f$ on the interval $[0,t]$ given by
\begin{equation}\label{GoodFunction}
\bar f(s) = 3a\left(s - \frac{s^2}{2t}\right) ; \ \ s\in[0,t].
\end{equation}
Note that $\bar f$ is not linear. This is somewhat surprising as
optimal Brownian paths which get to a large distance are roughly
linear; as can be seen by trying to minimise the Dirichlet energy of
functions going from 0 to a far away point $x$ in a given amount of
time $t$. (Indeed, optimal paths in the usual KPP equation are roughly
linear).

To identify the optimal trajectories here, recall that the
'probability' of a given path $f$ is roughly
$\exp( - (1/2) \int_0^t \bar f'(s)^2 ds )$. (This can be made rigorous
for instance using the theory of large deviations, see Schilder's
theorem.). Hence the function $\bar f$ is obtained as the minimiser of
the Dirichlet energy subject to a constraint:
$$ \bar f \equiv \arg\min \left\{ \frac12 \int_0^t \bar f'(s)^2ds: \text{ subject to } \int_0^t \bar f(s)ds = at^2\right\}.$$ 
By a standard calculus of variations argument,
(i.e., $\bar   + \eps \phi$ has a bigger Dirichlet norm than $\bar f$ for any function $\phi$ with $\int \phi  = 0$) we deduce that $\int_0^t \bar f '' \phi = 0$ for any such function after integration by parts. We deduce $\bar f '' $ is constant. So $\bar f$ is a parabola, of the form 
$$
\bar f(s) = bs + cs^2
$$
(there is no constant term as $f(0) = 0$. We claim that an optimal trajectory must further satisfy $\bar f'(t) = 0$. This can be justified on a heuristic level but can also be taken as an ansatz  otherwise; the fact that the upper bound and lower bound match up then justify it a posteriori. 

At any rate, this leads to the equation $2ct + b = 0$. Finally, plugging $\int_0^t \bar f(s) ds = at^2$ gives the value of the coefficients: $b = 3a$, $c = -3a/(2t)$. We then find
\begin{equation}\label{cost}
\int_0^t \bar f'(s)^2 ds = 3a^2 t
\end{equation}
consistent with the above lemma. The proof of the lower bound relies
crucially on the identification of the function $\bar f$ above. Indeed
our strategy will be to show that there are particles at the desired
distance $x = \gamma t^{3/2}$ by also requiring that $\theta_s$ stays
close to the time-reversal of $\bar f$. In particular, this will
explain the requirement on
$\theta_0 = \bar f (t) = 3at/2 = (\sqrt{2} /2) t$. (Note that this is
only half of the maximal value of $\theta^i_t$ among all particles
$i \in I_t$, since the particles $\theta^i$ are performing standard
BBM).  The particle will then reach its final position,
$x = \gamma t^{3/2}$ by moving linearly \emph{on the correct
  timescale}, that is the timescale which turns $X$ into a Brownian
motion. Consequently, the position at time $s$ of a particle ending up
at $\gamma t^{3/2}$ will be approximately given by
$$
X_s = W\left( \int_0^s \theta_u du\right) \approx \mu \int_0^s \theta_u du ,
$$
where $\mu = \gamma t^{3/2} / \left( \int_0^t \theta_u du \right) = \gamma t^{3/2} / (at^2) = \gamma t^{-1/2} /a$. 
Therefore we find,
$$
X_s \approx \mu \int_0^s \bar f(t-u) du  
$$
and making the relevant calculation,
\begin{equation}
\label{OptTraj}
X_s \approx \frac{3 \gamma}{2} \left( s \sqrt{t} - \frac{s^3}{3 t^{3/2}}\right). 
\end{equation}
It would be interesting to know whether the above guesses can give rise to a simplified and purely analytic proof of the main theorem of this paper, in the spirit of the recent PDE proof of Bramson's logarithmic delay in the KPP equation by Hamel, Nolen, Roquejoffre and Ryzhik \cite{HNRR}.  

\subsection{Open questions} 
From a probabilistic perspective, this raises several questions of
interest. Do particles near the maximum have a trivial correlation
structure, as in the usual branching Brownian motion? (In the
terminology of spin glasses, this would correspond to a 1-step replica
symmetry breaking for the associated Gibbs measure). Secondly, can the
shape of the front be described in more detail? We believe that the
effective size of the front (say, the spacing between the first and
third quartiles of $S(t,x)$, or some other function of $u$) does not
stay of order 1 as in the case of the usual branching Brownian
motion. Instead we believe that the front spreads over time with a
spacing of size roughly $\sqrt{t}$. Supporting evidence for this comes
from the proof of the lower bound, where a fluctuation of size
$\sqrt{t}$ is inherent in the identification of relevant
trajectories. Scaling by $\sqrt{t}$, does one obtain a limiting shape
for the front?

\section{Proof of lower bound for local equation}\label{sec:lower-bound}


For this we rely on our understanding of the optimal trajectories in
the previous section. Fix $\gamma<\gamma_0$ and set $a$ such that
$6a^3 = \gamma^2$, so $a<a_0$ and $\phi(a) = M(\gamma) >0$. Set
$$
\bar f(s) = 3a\left(s - \frac{s^2}{2t}\right), 
$$
and set $f$ to be a $C^2$ function which is always greater than $1/2$ and which coincides with $\bar f (t-s)$ on $[0, t-1]$. Thus for $s \le t-1$, 
\begin{equation}\label{f}
f(s) = \bar f(t-s) =   \frac{3a}{2} \left(t - \frac{s^2}t\right).
\end{equation}
Now set
\begin{align*}
M_s &= ( t-s+1/4)^{3/4}, \ \ 0 \le s\le t  \\ 
m & = 10 t^{1/2} \\
 \tau &= at^2  \\
 \mu& = \frac{\gamma t^{3/2}}{\tau}. 
\end{align*}
Recall that by choice, $\int_0^t \bar f(s)ds = at^2 = \tau$. 
Let $\Omega = C([0,\infty), \R)$ be the set of continuous trajectories, equipped with the Borel $\sigma$-field induced by the topology of uniform convergence on compacts. For $\theta \in \Omega$ set 
$$
J_s = \int_0^s \theta_s ds; \ \ J = J_t
$$

\subsection{Good events} 
Our goal will be to estimate $u(t,x, \theta)$ for values of $x$ close to $x= \gamma t^{3/2}$ where $\gamma$ is as above of $\theta$ satisfies: $| \theta - 3at/2| \le 1$. Hence we will assume that our initial condition is such that there is initially one particle at $x_0, \theta_0$ with
\begin{equation}\label{IC}
|x_0| \le 1; \ \ \left|\theta_0 - \frac{3a}{2} t \right| \le 1.
\end{equation}
We introduce the event  
\begin{equation}\label{Ai}
\cA = \left\{  \theta \in \Omega: \left|  \theta_s -  f(s) \right| \le M_s \text{ for all } s \le t  \right\} \cap \{ \tau - t \le J \le \tau \}.
\end{equation}
In other words the event $\cA$ is the event that $\theta $ remains not too far away (at most $M_s$) from the function $ f$, throughout the interval $[0,t]$, and has a total integral which is smaller than that of $f $ but by no more than $O(t)$. Recall that the function $ f$ is, apart from an unimportant additive constant $(1/2)$, the optimal trajectory identified in the previous section, which will guarantee that $\int_0^t \theta_s ds \approx at^2 = \tau$ and satisfies $\theta_t \le 1$. The error bound $M_s$ is decreasing from approximately $O(t^{3/4})$ initially to $M_t < 1/2$ at the end. In particular, note that if $\cA$ holds, then $0\le \theta_t \le 1$ and moreover $\theta$ never hits zero on $[0,t]$, and so there is no difference between $\int_0^t \theta_s ds$ and $\int_0^t |\theta_s|ds$. (In particular, the following proof works for both Dirichlet or Neumann boundary conditions.)

We introduce a second event $\cB$ which deals with the $W$ coordinate, and which is defined as follows. First observe that given 
Define the event $\cB$ by
\begin{equation}
\label{Bi}
\cB: = \left\{ W \in \Omega: \sup_{u \le \tau}   W_s  -  s \mu   \le m ,    \sup_{u\in [\tau - t, \tau] }\left|  W_{u} - \gamma t^{3/2} \right|\le m  \right\} 
\end{equation}
for some small enough $\delta$. In other words, the event $\cB$ is that the Brownian path progresses linearly towards its target position $\gamma t^{3/2}$ up to time $\tau$,
is always below the corresponding straight line (shifted by about $m = 10\sqrt{t}$), and lies within $m = 10 \sqrt{t}$ of that target throughout the interval $[\tau - t, \tau]$. 

Now return to the branching Brownian motion $(\theta^i_t, W^i_t)$ from the previous section. For $i \in \tilde I_t$, set 
 $$
 J^i_t = \int_0^t \theta^i_{s,t}ds,
 $$
 and set $K^i(\cdot)$ to be the cad lag inverse of $J^i_\cdot$. We will sometimes write $J^i$ for $J^i_t$ in order to lighten to the notations. Set $\tilde W^i(s) = X^i(K^i(s))$, and note that by Dubins--Schwarz theorem, for a fixed $i \in \tilde I_t$,  $\tilde W^i$ is just a Brownian motion over $[0, J^i]$.  For $i \in \tilde I_t$, let $\cA_i = \{ \theta_{\cdot, t}^i \in \cA\}$ and $\cB_i = \{ W^i \in \cB \}$. We shall consider the good event $\cG_i = \cA_i \cap \cB_i$ and set
$$
Z= \sum_{i \in I(t) } 1_{   \cG_i  }
$$
the number of particles which satisfy this good event. Note that if with high probability there is a particle satisfying $\cG_i$, (i.e. if $Z >0$) then at time $t$ the position of this particle is $ \gamma t^{3/2} (1+ o(1))$. Indeed, the position $X_t^i$ of particle $i$ at time $t$ is by definition $\tilde W_{J^i}^i $. On $\cA_i$, $|J^i - \tau | \le t$ so if $\cB_i$ also holds, $|\tilde W^i_{J^i} - \gamma t^{3/2} | \le m = 10 t^{1/2}$.  
Thus if $Z>0$ there is a particle such that $\cA_i \cap \cB_i$ holds and hence the maximal particle is greater than $ct^{3/2} - m$. We conclude using the McKean representation.

Hence it suffices to prove that $Z>0$ with probability tending to 1 as $t \to \infty$. To this end we use the Payley--Zigmund inequality:
\begin{equation}\label{PZ}
\P( Z>0) \ge \frac{ \E( Z)^2}{ \E(Z^2)}
\end{equation}
We will thus compute the first and second moment of $Z$, and show that $\E(Z^2) \le C \E(Z)^2$. This will show that $\P(Z>0) \ge p$ for some uniform $p>0$. A simple argument will then show that in fact we can bound this probability from below by something arbitrarily close to 1.

\subsection{First moment of $Z$}

We will first establish the following lower bound on the first moment of $Z$.

\begin{prop}\label{P:EZ}
\begin{equation}
\label{EZ}
\E(Z) \ge \frac{C}{t^{5/2}} \exp\left( t \left( 1 - \frac{3a^2}{2} - \frac{\gamma^2 }{2a}  \right) \right) .
\end{equation}
\end{prop}
\begin{rmk}
Note that, up to a polynomial factor, this matches the upper bound in \eqref{ubExp}. This polynomial term will be tracked carefully over the course of the proof, to match the upper bound on the second moment later on, so that the bound in \eqref{PZ} is bounded away from zero.
\end{rmk}

\begin{proof}
By the many-to-one Lemma,  $\E(Z) = e^t \P(\cA \cap \cB)$.  
Observe that $\cA$ and $\cB$ are independent, so we can estimate $\P(\cA)$ and $\P(\cB)$ separately. 

Let $\Q$ denote the law of Brownian motion (started from $\theta_0 = f(t) = 3at/2$ with drift $-\bar f(s)$, that is, the distribution on $C[0,t]$ of the process $(\theta_s - \bar f(s), s \le t )$ under $\P^{\theta_0}$. Then by Girsanov's theorem,
\begin{align*}
\P( \cA) & = \Q \left(  |\theta_s| \le M_s \text{ for all } s\le t ; - t \le J \le 0 \right)  \\
& = \E^0 \left ( 1_{\{ |  \theta_\cdot |\le M_\cdot , -t \le J \le 0  \} } \exp ( - \int_0^t  \bar f'(s) d\theta_s - \frac12 \int_0^t  \bar f'(s)^2 ds ) \right).
\end{align*}
Now, by stochastic integration by parts (or It\^o's formula), since $ \bar f$ is differentiable and is thus a finite variation process, we deduce that it has zero cross-variation with the Brownian motion $\theta$, and hence
\begin{align*}
 \int_0^t \bar f'(s) d\theta_s   &=  [ \bar  f'(t) \theta_t -  \bar f'(0) \theta_0]  - \int_0^t \theta_s  f''(s) ds \\
 & = \frac{3a}t \int_0^t \theta_s ds + O(1)
\end{align*}
since $\bar f$ satisfies $\bar f'(0) = 0$, $\theta_t = O(1)$ on $\cA$ and $\bar f'(t) = O(1)$, and $f''(s) = (-3a/t)$ for all $s\le t$. Therefore, using definition of the event $\cA$, and the relation \eqref{cost},
\begin{align}
\P(\cA) &= \exp \left( - \frac12 \int_0^t f'(s)^2ds \right) \E^0  \left ( 1_{\{ |  \theta_\cdot |\le M_\cdot , -t \le J \le 0  \} } \exp \left( - \frac{3a}t \int_0^t \theta_s ds\right) \right) \nonumber \\
& \ge \exp \left( - \frac{3a^2}{2}t \right) \P^0 \left(   |  \theta_\cdot |\le M_\cdot , -t \le J \le 0 \right) \exp( - 3a)\label{PA1} 
\end{align}
Now, we will need the following Lemma:
\begin{lemma}
\label{L:theta_int_dev}
$$
\P^0 \left(  |  \theta_\cdot |\le M_\cdot , -t \le J \le 0 \right)  \ge \frac{C}{t} 
$$
\end{lemma} 

\begin{proof}
Let $F_{a,b}$ be the event that $|\theta_s| \le M_s$ for all $a\le s\le b$. We also introduce the stronger event $F_{a,b}^+$ that $|\theta_s| \le M_s/2$ for all $a\le s\le b$

We will denote $\bar \theta_s = \theta_{t-s}$ and also introduce $\bar F_{a,b} = \{|\bar \theta_u |\le  \bar M_u \text{ for all } u \in [a,b]\}$ where $\bar M_u = M_{t-u}$. Fix $s_0 =  \lambda (\log t)^2$ where $\lambda$ is a big constant to be chosen later.  We will work conditionally given $\bar \Theta_{s_0} = \sigma( \bar \theta_s, s\le s_0)$ and assume that $F^+_{t-s_0, t}$ holds. Note that, conditionally on $\bar \Theta_{s_0}$, the process $(\theta_s, 0 \le s\le t - s_0)$ is a Brownian bridge of duration $t-s_0$ from 0 to $z = \bar \theta_{s_0}$, a distribution which we will denote by $\BB{0}{z}{t-s_0}$. Its time-reversal is hence a Brownian bridge of duration $t-s_0$, from $z$ to 0.

Observe that a Brownian bridge from $0$ to $z$, of duration $t-s_0$ , when restricted to the interval $[0,3t/4]$, is absolutely continuous with respect to a Brownian motion started from 0, and furthermore the density is uniformly bounded by a constant $C$ (see e.g. (6.28) in \cite{KS}), since $z = o(\sqrt{t})$ on $F_{t-s_0, t}^+$. 
 Therefore, for $s\le 3t/4$,
\begin{align}
\BB{0}{z}{t} ( F_{0, 3t/4}) & \le  C \P^0(F_{0,3t/4}) \le \P( \sup_{s \le 3t/4} \theta_s \ge c t^{3/4}) \le 2 e^{- c t^{1/2}} \label{hitM0}
\end{align}
by the reflection principle and \eqref{tail}.

Denote $\bar E_s = \bar F^c_{s,2s}$.   
Note that once again the law of $\BB{z}{0}{t-s_0}$, restricted to $[0, t/2]$, is absolutely continuous to a Brownian motion started from $z$ and the density is uniformly bounded. Hence for $s\le t/4$, by scaling, since $|z| \le (1/2) \bar M_{s_0}$ on $F_{t-s_0, t}^+$,
\begin{align}
 \BB{z}{0}{t-s_0}(\bar \theta_u = \bar M_u \text{ for some } u \in [s, 2s]) &\le C \P^z ( \bar \theta_u = \bar M_u \text{ for some } u \in [s, 2s]) \nonumber \\
&\le  C \P^0 ( \sup_{u \le 2} \theta_u \ge C' s^{1/4} ) \nonumber \\
& \le  C e^{- C'  s^{1/2}} \label{hitM}
\end{align}
by the reflection principle and \eqref{tail} again. By combining \eqref{hitM0} with \eqref{hitM} we deduce that
\begin{equation}\label{hitM2}
 \P^0( F_{0, t - s_0} ^c | \bar \Theta_{s_0})  \le t^{-2}
\end{equation}
for $\lambda$ sufficiently large, on $F_{t-s_0, t}^+$. 

Now let $J_1 = \int_0^{s_0} \bar \theta_u du$ and let $J_2 = J - J_1 = \int_{s_0}^t \bar \theta_u du$. 
Note that
\begin{align*}
\P( -t \le J \le 0 ; F_{s_0,t} | \bar \Theta_{s_0}) & = \P(-t - J_1 \le J_2 \le -J_1 | \bar \Theta_{s_0})  - \P(  F^c_{s_0,t} | \Theta_{s_0}) .
\end{align*}
Note that, given $\bar \Theta_{s_0}$, $J_2$ is the integral of a certain Gaussian process, namely a Brownian bridge of duration $t- s_0$ from $\bar \theta_{s_0}$ to 0. This has mean $m_2 =  z ( t - s_0)/2$ and a computation similar to \eqref{variance_int} shows that it has a variance of order at least $C t^3$. 

Hence
$$
 \P(-t - J_1 \le J_2 \le -J_1 | \bar \Theta_{s_0})  = \P\left( \mathcal{N}(0,1) \in  \left[ \frac{-m_2 - t - J_1}{Ct^{3/2}}, \frac{ - m_2 - J_1}{Ct^{3/2}} \right] \right).
 $$
On the event we are considering, $0 \le m_2 \le t (\log t)^{3/4}$ and $0\le J_1\le s_0^{7/4} \le (\log t)^2$, so the interval in the right hand side has a size at least $O(1/\sqrt{t})$ and is located within $[-1,1]$ where the density of $\mathcal{N}(0,1)$ is bounded away from zero. Hence,
 $$
  \P(-t - J_1 \le J_2 \le -J_1 | \bar \Theta_{s_0})   \ge \frac{C}{\sqrt{t}}.
  $$

Consequently, taking expectations so as to remove the conditional expectation given $\Theta_{s_0}$, we deduce
that
\begin{align}
\P( - t \le J \le 0; F_{0,t}) & \ge \E(\mathbf{1}_{F_{t- s_0,t}^+} \P( -t \le J \le 0 ; F_{0,t - s_0} | \bar\Theta_{s_0}) )\nonumber \\
&\ge \E(\mathbf{1}_{F_{t- s_0,t}^+} [\P( -t \le J \le 0  | \bar\Theta_{s_0})  - \P (  F_{0,t - s_0} | \bar \Theta_{s_0}) ] ) \nonumber \\
&\ge \P( F_{t - s_0, t}^+ )  C ( t^{-1/2} - t^{-2} ) \label{hitM3} 
\end{align}

Note also that $\P(F_{t,t}^+) \ge C t^{-1/2}$ and, by following the argument in \eqref{hitM}, $\P(F_{t-s_0, t}^+ | F_{t,t}^+) \ge p$ for some $p>0$. It follows that  $\P( F_{t - s_0, t}^+ )  \ge C t^{-1/2}$. Plugging this into \eqref{hitM3} we get  
$$
\P( - t \le J \le 0; F_{0,t})  \ge Ct^{-1}
$$
as desired.
\end{proof}

Altogether, combining Lemma \ref{L:theta_int_dev} and \eqref{PA1}, we get
\begin{equation}
\label{PA}
\P(\cA)   \ge  \frac{C}{t} \exp \left( - \frac{3a^2}{2}t \right).
\end{equation}

We now turn to $\P( \cB )$, which is somewhat easier.  Recall the value of the drift  
$$\mu = \frac{ \gamma t^{3/2} }{\tau} = O( t^{-1/2}),$$
so that $\mu \tau = \gamma t^{3/2}$ (i.e., $\mu$ is the slope of the line involved in the definition of $\cB$). 
From Girsanov's theorem,
\begin{align*}
\P( \cB ) & =  \P\left( \sup_{u \le \tau}   (W_u - u \mu)  \le m ; \sup_{u \in [\tau - t, \tau] } |  W_{u} - \gamma t^{3/2}  | \le m\right) \\
& = \E^0 \left( 1_{\{  \sup_{u \le \tau}   (W_u - u \mu)  \le m ; \sup_{u \in [\tau - t, \tau] } |  W_{u} - \gamma t^{3/2}  | \le m \} } \exp ( - \mu  W_{\tau} - \frac12\mu^2 \tau )  \right)  \\
& \ge \exp ( -  \mu m  - \frac12 \mu^2  \tau )  \P^0 \left(  \sup_{u \le \tau}   W_u   \le m ; \sup_{u \in [\tau - t, \tau] } |  W_{u}   | \le m\right) \\
& \ge  O(1)  \exp (  - \frac{c^2 t }{2a} )  \P^0 \left( \sup_{u \le \tau}   W_u   \le m ; \sup_{u \in [\tau - t, \tau] } |  W_{u}   | \le m \right).
\end{align*}

\begin{lemma}
$$
 \P^0 \left( \sup_{u \le \tau}   W_u   \le m ; \sup_{u \in [\tau - t, \tau] } |  W_{u}    | \le m \right) \ge c m^3 \tau^{-3/2}
$$
\end{lemma}
\begin{proof}
We can use the reflection principle to compute this. Indeed, we know that for $0\le a \le b$, letting $S_T = \sup_{s\le T} W_s $,
$$
\P_0(S_T \ge b, W_T \le a) = \P( W_T \ge 2b - a) =\int_{2b -a}^\infty  \frac1{\sqrt{2\pi T}}  e^{ - x^2 / (2T)} dx
$$
and hence the joint density of $(S_T, W_T)$ at the point $0 \le a\le b$ is, after differentiating twice the above expression, is
\begin{equation}\label{reflection}
\frac{2(2b- a)}{2\sqrt{2\pi} T^{3/2}} e^{- (2b- a)^2/(2T)}
\end{equation}
If $T \ge C m^2$, by integrating, we find
\begin{align*}
\P(  \sup_{u \le T}   W_u   \le m ; W_{T} \in [- 3m/4, - m/4] ) & \ge \int_{-3m/4}^{-m/4}\int_0^m \frac{C(2b-a)}{T^{3/2}} e^{ - (2b- a)^2/(2T)} db da\\
& \ge \frac{C}{T^{3/2}}  \int_{-3m/4}^{- m/4} \int_0^m (2b-a ) db da  \ge \frac{C m^3}{T^{3/2}}. 
\end{align*}
We apply this result with $T = \tau - t$ and then apply the Markov property at time $T$, to find that
\begin{align*}
 & \P^0 \left( \sup_{u \le \tau}   W_u   \le m ; \sup_{u \in [\tau - t, \tau] } |  W_{u}    | \le m \right)  \\
   &\ge  \P(  \sup_{u \le T}   W_u   \le m ; W_{T} \in [- 3m/4, - m/4] ) \P^0 
 (\sup_{u \in [0,t]} W_u \le m /4)\ge  \frac{C m^3}{T^{3/2}} p
\end{align*}
for some $p>0$, by scaling. The result follows. 
 \end{proof}

 Putting things together, we find
 \begin{equation}\label{PB}
 \P(\cB ) \ge O(1) t^{-3/2}   \exp (  - \frac{\gamma^2 t }{2a} ) ,
 \end{equation}
Now, Proposition \ref{P:EZ} follows by combining \eqref{PA} and \eqref{PB}.
\end{proof}

\subsection{Second moment of $Z$}

In this section we prove the following upper bound on the second moment of $Z$.

\begin{prop}
\label{P:Z2}
$$
\E(Z^2) \le \frac{C}{t^5} \exp( 2t ( 1 - \frac{3a^2}{2} - \frac{\gamma^2 }{2a}  ) ) .
$$
\end{prop}

To compute the second moment we use a modified version of the many-to-one Lemma, which can be called the many-to-two lemma. We begin with a useful definition for what follows. 

\begin{definition}
Let $B_1$ be a real Brownian motion and $T\ge 0$ a possibly random time. We say that $B_2$ branches from $B_1$ at time $T$ if there exists another Brownian motion $W$, independent of $B_1$ and $T$, such that 
$$
B_2(u) = 
\begin{cases}
B_1(u) & \text{ for } u \le T;\\ 
B_1(u) + W(u-T) & \text{ for } u \ge T.
\end{cases}
$$
\end{definition}

This definition is in fact symmetric: if $B_2$ branches from $B_1$ at time $T$, then $B_1$ branches from $B_2$ at time $T$. We will sometime simply say that $B_1$ and $B_2$ branch from each other at time $T$. 

\begin{lemma}
\label{L:many_to_two}
Let $F$ be a measurable functional on $C[0, t]$, and set 
$$
Z = \sum_{i \in I_t} F(X_{s,t}^i, 0 \le s \le t).
$$ 
Let $T$ be an exponential random variable with parameter $2$, and let $B_1, B_2$ be Brownian motions branching from each other at time $T$. 
Then we have
$$
\E(Z^2) = e^{2t} \E [ e^{T\wedge t} F(B_1(s), 0 \le s \le t ) F(B_2(s), 0 \le s \le t)].
$$
\end{lemma}

\begin{proof} 
See \cite{manytofew} for a slightly more general result. 
\end{proof}

We will use it as follows. For a fixed $0< s < t$, let $(B_1, B_2)$ denote  Brownian motions branching from each other at time $T=s$. (Technically we should indicate the dependence on $s$ within the notation, but we will avoid doing this in order to ease readability). We then have, if $F(X) = \indic{ X\in \cA}$ for some Borel set $\cA$ on $C[0,t]$ 
\begin{equation}\label{many2}
\E(Z^2) = \E(Z) + 2 \int_0^t e^{2(t-s)}  e^s \P(B_1\in \cA, B_2 \in \cA) ds. 
\end{equation}
This can be interpreted as follows: write $Z^2 = \sum_{i,j} \indic{X_i, X_j \in \cA}$.  If we decompose on the time at which the particles $i$ and $j$ have their most recent common ancestor, say $s$, then there are $e^s$ expected potential ancestors at generation $s$, and each produces $e^{t-s}$ expected descendants at generation $t$. Hence the number of pair of descendants descending from a given individual at generation $s$ is $e^{2(t-s)}$ if we order them, and each pair is counted twice (hence an additional factor 2). 

Hence in view of \eqref{many2}, our task will be to show that if $\cA$ and $\cB$ are the events defined in \eqref{Ai} and \eqref{Bi}, and $\cG= \cA \cap \cB$, then 
\begin{equation}
\label{goal}
\int_0^t e^{2(t-s)}  e^s \P(B_1\in \cG, B_2 \in \cG) ds \le \frac{C}{t^6} \exp( 2t ( 1 - \frac{3a^2}{2} - \frac{\gamma^2 }{2a}  ) ) .
\end{equation}
This will imply Proposition \ref{P:Z2} as it is clear that $\E(Z) \to \infty$ and thus $\E(Z)\le \E(Z)^2$ for $t$ large enough. 

We thus turn to the 
\begin{proof}[Proof of \eqref{goal}]
Let $0 \le s \le t$. We begin by spelling out the event $\{B_1, B_2 \in \cG\}$ more explicitly. Take two independent real Brownian motions $\theta$ and $W$. We can then introduce a proces $(\theta'_u, 0\le u \le t)$ which branches from $\theta$ at time $s$, and a process $W'$ which is a Brownian motion branching from $W$ at time $J_s$, where 
$$
J_s = \int_0^s \theta_u du = \int_0^s \theta'_u du.
$$
Thus $W'_u = W_u$ for $u \le J_s$ and $W'_u = W_{J_s} + \tilde W_{u - J_s}$ for $u \ge J_s$, where $\tilde W$ is independent from $(\theta, W)$. 

\textbf{Event $\cB$}. The event $\{B_1, B_2 \in \cG\} $ can be reformulated as  $\{ \theta, \theta' \in \cA \} \cap \{ W, W' \in \cB\}$. 
We will first condition on the entire processes $\theta, \theta'$ and compute the conditional probability that $(W, W') \in \cB$. Recall our notation $J_s = \int_0^s \theta_u du$. For this computation, as before, the key result we use is Girsanov's theorem. The joint law of $(W_u - u \mu, W'_u - u \mu, u \le \tau)$ has density 
\begin{align}
Z_B &= \exp\left( \left\{- \mu W(J_s) - \frac12 \mu^2 J_s \right\}  - \{\mu(W(\tau) - W(J_s)) + \frac12 \mu^2 (\tau - J_s)\} \right.\\
& \ \   \ \ \ \ \ \ \  \left. -\left\{ \mu(W'(\tau) - W'(J_s)) + \frac12 \mu^2 (\tau - J_s)\right\} \right).
\label{Girsanov2}
\end{align}
We have grouped the terms in brackets so that they are independent of one another (conditionally on $(\theta, \theta')$), but in fact we will reorder them in a convenient way later on. 
Consequently,
\begin{align*}
\P( W, W' \in \cB | ( \theta, \theta') ) & = \E\left( \left. Z_B \indic{  \sup_{u \le \tau}   (W_u, W'_u )  \le m ; \sup_{u \in [\tau - t, \tau] } |  W_{u}, W'_u  | \le m} \right| (\theta, \theta') \right) \\
& \le \E\left( \left. Z_B \indic{ \sup_{u \le J_s}W_{u} \le m} \indic{\sup _{u \in [J_s, \tau]}(W_u, W'_u ) \le m}  \indic{W_\tau, W'_\tau \ge - m } 
\right| ( \theta, \theta') \right)\\
&\le  \exp\left( - \frac12 \mu^2 \tau - \frac12 \mu^2 (\tau - J_s) \right) \times\\
& \times  \E \left(  e^{ \mu W(J_s) } \indic{W(J_s) \le m} e^{ - \mu W(\tau) - \mu W'(\tau) } \indic{\sup _{u \in [J_s, \tau]}(W_u, W'_u ) \le m}  \indic{W_\tau, W'_\tau \ge - m } 
\right)  
\end{align*}
Now, we need the following lemma:
\begin{lemma}\label{L:3/2}
For all $x, y >0$, 
$$
\P( \inf_{u\le 2T} W_u \ge 0 ; W_{2T} \in [ 0, y]| W_0 = x) \le \frac{Cx y^2 }{T^{3/2}}. 
$$
where $C$ is independent of $x$, $y$.
\end{lemma}

\begin{proof}
The cases where $x\le \sqrt{T}$ and $x \ge \sqrt{T}$ can be treated similarly. For simplicity we focus on the case $x \le \sqrt{T}$ which is slightly more interesting. We first note that $\P(\inf_{u\le T} W_u \ge 0  | W_0 = x) \le c x/\sqrt{T}$ by well known (and very simple) estimates. By the Markov property of Brownian motion at time $T$, conditionally given $\inf_{u\le T} W_u \ge 0  $ and $W_{T}=z$, the probability that $\inf_{u\in [T, 2T]} W_u \ge 0$ and $W_{2T} \in [0, y]$ can be expressed as  
\begin{align}\label{32}
\P_x( \inf_{u\in [0, 2T]} W_u \ge 0  ; W_{2T} \in [0,y] | W_{T} = z; \inf_{u \le T} W_u \ge 0 ) & = \P_0( \sup_{u \le T} W_u \le z; W_{T} \in [z-y, z] )
\end{align}
The latter can be computed using the reflection principle. Indeed, we know that for $0\le a \le b$, letting $S_T = \sup_{s\le T} W_s $,
$$
\P_0(S_T \ge b, W_T \le a) = \P( W_T \ge 2b - a) =\int_{2b -a}^\infty  \frac1{\sqrt{2\pi T}}  e^{ - x^2 / (2T)} dx
$$
and hence the joint density of $(S_T, W_T)$ at the point $0 \le a\le b$ is, after differentiating twice the above expression, is
$$
\frac{2(2b- a)}{2\sqrt{2\pi} T^{3/2}} e^{- (2b- a)^2/(2T)}
$$
Integrating over $0\le b \le z$ and $0 \le z-y \le a\le b$ gives us, after applying Fubini's theorem,  making a change of variable, 
\begin{align*}
\P_0( \sup_{u \le T} W_u \le z; W_{T} \in [z-y, z] ) & = \int_{z-y}^z {da} \int_{a}^z \frac{(2b -a)}{ \sqrt{2\pi  } T^{3/2} } e^{- (2b- a)^2/(2T)}db \\
& \le \frac{C}{T^{3/2}}  \int_{z-y}^z {da} \int_a^z (2b-a) db\\
& = \frac{C}{T^{3/2}}  \int_{z-y}^z z(z-a) {da} = \frac{C}{T^{3/2}}  zy^2.
\end{align*}
Taking expectations in \eqref{32} this yields
\begin{align*}
\P_x( \inf_{u\in [0, 2T]} W_u \ge 0  ; W_{2T} \in [0,y]  ) &\le \frac{Cxy^2 }{T^2} \E_x ( W_T | \inf_{u \le T} W_u \ge  0 )   \\
& \le \frac{Cxy^2}{T^{3/2}}. 
\end{align*}
Indeed, the position at time $T$ of $W_T$, given  $\inf_{u \le T} W_u \ge  0$, is dominated stochastically from above by a three-dimensional Bessel process started from $x$ at time $T$, whose expectation is $\E_x(R_T) \le c_3 \sqrt{T}$ by scaling, for some constant universal constant $c_3$, since we assumed $x \le \sqrt{T}$. The case $x\ge \sqrt{T}$ is similar (but easier). 
\end{proof}

We condition on $W(J_s)$ and apply Lemma \ref{L:3/2}. Since $W(\tau) \ge - m$ and $\mu m \le O(1)$, we thus obtain 
\begin{align*}
& \E \left(e^{ - \mu W(\tau) - \mu W'(\tau) } \indic{\sup _{u \in [J_s, \tau]}(W_u, W'_u ) \le m}  \indic{W_\tau, W'_\tau \ge - m } | W(J_s)  \right)  \\
& \le  O(1) \P \left( {\sup _{u \in [J_s, \tau]}(W_u, W'_u ) \le 1} ;  {W_\tau, W'_\tau \ge - m } | W(J_s)   \right)  \\
& \le O(1) \frac{(m+|W(J_s)|)^2 m^4}{(\tau - J_s)^3}    \le O(1) \frac{(m^2+W(J_s)^2) m^4}{(\tau - J_s)^3}.
\end{align*}
Taking the expectation again, we find
\begin{align}
\P( W, W' \in \cB | ( \theta, \theta') ) & \le O(1)  \exp\left( - \frac12 \mu^2 \tau - \frac12 \mu^2 (\tau - J_s) \right)\frac{m^4}{(\tau - J_s)^3}\times \\
& \ \ \ \times \E( e^{\mu W(J_s)} (m^2 + W(J_s)^2) \indic{W(J_s) \le m} | ( \theta, \theta') )\nonumber  \\
& \le O(1)  \exp\left( - \frac12 \mu^2 \tau - \frac12 \mu^2 (\tau - J_s) \right)\frac{m^4}{(\tau - J_s)^3}  (m^2 + \E( W(J_s)^2 | ( \theta, \theta')) )\nonumber \\
& =  O(1)  \exp\left( -  \mu^2 \tau + \frac12 \mu^2 J_s  \right)\frac{m^4}{(\tau - J_s)^3} (m^2 + J_s) = : B. \label{Bdom}
\end{align}

\medskip \noindent \textbf{Event $\cA$}. We now turn to the event $\theta, \theta' \in \cA$. As in \eqref{Girsanov2}, we have that the joint law of $(\theta_u - f(u), \theta'_u - f(u), 0 \le u \le t) $ has a density (with respect to a pair of driftless Brownian motions branching from each other at $s$) which is
\begin{align}
Z_A & = \exp \left( - f'(s) \theta_s -\frac{3a}{t} \int_0^s \theta_u du - \frac12 \int_0^t f'(u)^2 du + \right. \nonumber \\
& \ \ \ \ \ \ \ \ \ \ \left. + \{ f'(s) \theta_s - \frac{3a}{t} \int_s^t \theta_u du  - \frac12 \int_s^t f'(u)^2 du \}  + \right.\nonumber \\
& \ \ \ \ \ \ \ \ \ \ \left. + \{ f'(s) \theta'_s - \frac{3a}{t} \int_s^t \theta'_u du  - \frac12 \int_s^t f'(u)^2 du \} \right) \nonumber \\
& \le \exp \left(  - \frac12 \int_0^t f'(u)^2 du -   \int_s^t f'(u)^2 du  +  \theta_s f'(s)  + \frac{3a}{t } \int_0^s \theta_u du\right) . \label{Girsanov3}
\end{align}
where the inequality holds on the event $\theta, \theta' \in \cA$. 
Furthermore, $f'(s) = -3as/t $, $|\theta_s |\le M_s \le t^{3/4}$, hence $\int_0^s \theta_u du \le C s t^{3/4}$, so that 
$$
Z_A\le  z^+_A : = \exp \left(  -  \int_0^t f'(u)^2 du + \frac12  \int_0^s f'(u)^2 du  + \frac{Cs}{t^{1/4} }\right)
$$
Note that $z^+_A$ is nonrandom. 
Also, on the event $\theta, \theta' \in \cA$, and while $s\le t /2$, the random variable $B$ on the right hand side of \eqref{Bdom} is bounded above by 
\begin{align*}
B& \le O(1) \frac{m^4}{t^6} ( t + st)   \exp\left( -  \mu^2 \tau + \frac12 \mu^2 J_s  \right)\\
& \le O(1) \frac1{t^3} ( 1+ s) \exp\left( -  \mu^2 \tau + \frac12 \mu^2 J_s  \right).
\end{align*}

Therefore, we may rewrite, letting $j(s) = \int_0^s f(u) du$, 
\begin{align*}
\P ( \theta, \theta' \in \cA; W, W' \in \cB) &  \le \E ( B \indic{\theta, \theta' \in \cA} ) \\
& \le O(1) \frac{s+1}{t^3}  \exp\left( -  \mu^2 \tau \right)  \E\left( e^{\frac12 \mu^2 J_s} \indic{\theta, \theta' \in \cA} \right) \\
& \le O(1) \frac{s+1}{t^3}   \exp\left( -  \mu^2 \tau  + \frac12 \mu^2 j(s) \right) \E \left( z^+_A \indic{ | \theta_\cdot|, |\theta'_\cdot| \le M_u; -t \le J, J' \le 0 } e^{\frac12 \mu^2 J_s}\right)  \\
& \le O(1) \frac{s+1}{t^3}   \exp\left( -  \mu^2 \tau  + \frac12 \mu^2 j(s) \right) z^+_A  \exp\left(\frac12 \mu^2 \frac{3as}{t^{1/4}} \right)\times\\
& \ \ \ \ \ \ \ \ \ \times \P(  - t \le J, J' \le 0 ; | \theta_\cdot|, |\theta'_\cdot| \le M_\cdot)
\end{align*}
where we have used that (as above) $J_s \le \frac{3as}{t^{1/4}}$ on $|\theta_\cdot|, |\theta'_\cdot| \le M_\cdot$. Since $\mu^2 =O(1/t)$, we deduce
$$
\P ( \theta, \theta' \in \cA; W, W' \in \cB) \le O(1) \frac{s+1}{t^3}  \exp\left( -  \mu^2 \tau  + \frac12 \mu^2 j(s) \right) z^+_A  \P(   - t \le J, J' \le 0 ;  | \theta_t|, |\theta'_t| \le 1 ).
$$
We will need the following lemma:

\begin{lemma}
\label{L:bridge} For $x,j,a\in \R$,  
$$
\P_x( j - a\le J \le j ; |\theta_t| \le 1 ) \le \frac{C a }{(t +1 )^{2}}
$$
for some constant $C$ independent of $x$, $j$, $a$ and $t$.
\end{lemma}

\begin{proof}
Observe first that the position $\theta_t$ is normally distributed with mean $x$ and variance $t$, so its density is uniformly bounded by $Ct^{-1/2}$ and $\P_x( |\theta_t| \le 1) \le C t^{-1/2}$. 

When we condition on the value of $\theta_t = z$, $J$ is a normal random variable, as the integral of a certain Gaussian process (namely, a Brownian bridge of duration $t$ from $x$ to $z$) with a certain mean and a variance at most $t^3/3$. Consequently the probability density function of $J$ is uniformy bounded by $C /(t+1)^{3/2}$. Integrating over the interval $[j-a, j]$ gives 
$$
\BB{x}{z}{t} ( J \in [j-a, j]) \le \frac{C a }{(t +1 )^{3/2}},
$$
uniformly in $z$. Hence 
$$
\P_x( j - a\le J \le j ; |\theta_t| \le 1 )\le  \frac{C a }{(t +1 )^{3/2}} \P_x ( |\theta_t| \le 1)
\le \frac{Ca}{t^2}.
$$
Since this probability is also trivially bounded by 1, we get the desired upper bound.
\end{proof}

From Lemma \ref{L:bridge} it follows, after conditioning by $(\theta_u, u \le s)$, that 
$$ 
\P( - t \le J, J' \le 0 ;   |\theta_t|, |\theta'_t| \le 1) \le \frac{Ct^2 }{(t+1-s)^4}.
$$
Consequently,
\begin{align*}
\P ( \theta, \theta' \in \cA; W, W' \in \cB) & \le O(1) \frac{s+1}{t^3}   \exp\left( -  \mu^2 \tau  + \frac12 \mu^2 j(s) \right) z^+_A  \frac{t^2}{(t+1-s)^4}\\
& \le \frac{O(1)(s+1) }{t (t+1 - s)^4}   \exp\left( -  \mu^2 \tau  + \frac12 \mu^2 j(s) \right) z^+_A  .
\end{align*}

We return to the integral that we wish to estimate in \eqref{goal}, which is
$$
\int_0^{t} e^{2t} e^{-s} \P( B, B' \in \cG) ds.
$$
On the interval $[0, t/2]$ we may write that $t+1 -s \ge t/2$ so we deduce, after writing out the various terms,
\begin{align}
e^{2t} \int_0^{t/2}  e^{-s} \P( B, B' \in \cG) ds& \le e^{2t - \mu^2 \tau - \int_0^{t} f'(u)^2 du} \frac{O(1)}{t^5} \int_0^{t/2}  e^{-s} (s+1) e^{\frac12 \mu^2 j(s) + \frac12 \int_0^s f'(u)^2du + Cs/t^{1/4}}ds \label{corbound}
\end{align}
Over the interval $[t/2, t]$, we use the crude bound $t+1-s \le 1$ and we get
\begin{align}
e^{2t} \int_{t/2}^{t}  e^{-s} \P( B, B' \in \cG) ds& \le e^{2t - \mu^2 \tau - \int_0^{t} f'(u)^2 du} \frac{O(1)}{t} \int_{t/2}^t  e^{-s} (s+1) e^{\frac12 \mu^2 j(s) + \frac12 \int_0^s f'(u)^2du + Cs/t^{1/4}}ds \label{corbound2}
\end{align}
Note that the exponential prefactor in front of the two integrals above is precisely $e^{2t (1 - 3a^2 - c^2/2a)}/t^5$, as desired in \eqref{goal}. Hence it suffices to show that the total integral above (from $0$ to $t$) is bounded. We compute the exponential term inside the integrand. Recall that for $s\le t-1$, $f(s) = (3a/2)( t - s^2/t)$, that $\mu = c t^{-1/2}/a$. Thus,
\begin{align*}
j(s)   = \int_0^s f(u)du & = \int_0^s \frac{3a}{2} (t - \frac{u^2}{t}) du = \frac{3a}{2}(ts  - \frac{s^3}{3t})
\end{align*}
so that
\begin{equation}\label{j}
\frac12 \mu^2 j(s) = \frac{3\gamma^2}{4a}( s - \frac{s^3}{3t^2}).
\end{equation}
Also,
\begin{align}\label{f'}
\int_0^s f'(u)^2 du & = \frac{9a^2}{t^2} \int_0^s u^2 du =  \frac{3a^2}{t^2}s^3.
\end{align}
Hence returning to the exponential term in the integral in the right hand side of \eqref{corbound} and \eqref{corbound2}, and making a change of variable $s = xt, x\in (0,1)$, 
$$
 \exp\left\{ -s + \frac12 \mu^2 j(s) + \frac12 \int_0^s f'(u)^2du + Cst^{-1/4} \right\} = \exp\left\{ t \psi (x)   +C  t^{3/4} x \right\} 
$$
where
$$
\psi(x) = x ( -1 + \frac{3\gamma^2}{4a}) + x^3 ( \frac{3a^2}{2} - \frac{\gamma^2}{4a}).
$$
Using the relation $\gamma^2 = 6 a^3 $ we see that all the cubic terms cancel, and we are left with
\begin{align*}
\psi(x) & =  - x(1 - \frac{9a^2}{2} )
\end{align*}
and note that since $a< \sqrt{2}/3$, $1- 9a^2/2 >0$. Thus returning to the integral in \eqref{corbound}, this becomes:
$$
\int_0^{t/2} (s+1) \exp\left( - s( 1- \frac{9a^2}{2}) +  {C s}{t^{-1/4} } \right) ds .
$$
It is easy to see that this is less than $C$ for some universal constant $C>0$. 

We deduce that
\begin{equation}
\E(Z^2) \le \frac{C}{t^5} \exp( 2t ( 1 - \frac{3a^2}{2} - \frac{\gamma^2 }{2a}  ) )
\end{equation}
which concludes the proof of Proposition \ref{P:Z2}. 
\end{proof}

\section{End of proof of Theorem \ref{T:toads}}\label{sec:conclusion-proof}


It is now not hard to finish the proof of Theorem \ref{T:toads}.
\begin{proof}[Proof of Theorem \ref{T:toads}]
  First suppose that $u_0(x, \theta) = H(x) \indic{\theta\in (0,1)}$.
  In this case we have already obtained the upper bound
  $\sup_{\theta_0 >0} \sup_{x> \gamma t^{3/2}} u(t,x, \theta_0) \to 0$
  in \eqref{conclUB} for $\gamma > \gamma_0$.

For the lower bound, combining Propositions \ref{P:EZ} and \ref{P:Z2}
and the Payley--Zygmund inequality \eqref{PZ} we see that for
$\gamma <\gamma_0$,
\begin{equation}
\label{liminf}
\liminf_{t \to \infty} \P^{x, \theta}( \max X_i(t) > \gamma t^{3/2}) \ge C
\end{equation}
for some $C>0$ and for  $  \theta_0 \in (\theta_-, \theta_+) $ and $x_0 \in (-1,1)$ where $
 \theta_\pm = \frac{3a }{2}t \pm 1,$
  as per \eqref{IC}. Supposer that initially there is one particle at $x_0 = 0$ and $\theta_0 = 3a t /2$.  Fix a large constant $T>0$ and condition on the population at time $T$ (given $\cF_T)$. Let $I$ be the set of particles $i \in \tilde I(T)$ such that $\theta^i_T \in [\theta_-, \theta_+]$ and $X^i_T \in (-1,1)$. For each particle $i \in I$, let $x_i = X^i_T$ and $\theta_i = \theta^i_T$. Then 
\begin{align*}
\P^{0,\theta_0} \left( \max_{i\in \tilde I_t} X_t^i \le c t^{3/2} | \cF_T\right)  \le \prod_{i \in  I} \P^{x_i,\theta_i}\left( \max_{i \in \tilde I_{T-t}} X_t^i \le \gamma (t - T)^{3/2} - x_i  \right)  
\end{align*}
so 
$$
\liminf_{t \to \infty} \P^{0,\theta_0}\left( \max_{i\in \tilde I_t} X_t^i \le c t^{3/2} | \cF_T\right) \le (1- C)^{| I | }
$$
Also, the random variable on the left hand side (before taking the liminf) is bounded by one, so by the dominated convergence theorem, 
$$
\liminf_{t\to \infty}  \P^{0,\theta_0}\left( \max_{i\in \tilde I_t} X_t^i \le \gamma t^{3/2} \right) \le \E( (1-C)^{|I|} )
$$
uniformly in $T$. Since it is clear that $|I | \to \infty$ in probability as $T \to \infty$, we deduce that 
\begin{equation}
\label{LB_heavyside}
\liminf_{t\to \infty}  \P^{0,\theta_0}\left( \max_{i\in \tilde I_t } X_t^i \le \gamma t^{3/2} \right) = 0,
\end{equation}
which completes the proof of Theorem \ref{T:toads} in the case of the Heavyside initial condition. 

We now turn to the general case of initial conditions subject to our assumptions, and use the general case of the McKean representation (Proposition \ref{P:mckean}) together with the result obtained above in the particular case of Heavyside data.

We first consider the upper bound on the speed. Fix $\gamma > \gamma_0$ and let $\gamma_0 < \gamma' < \gamma$ and let $\delta = \gamma' - \gamma_0$. Note that for $x \ge \gamma t^{3/2}$, 
$$
\prod_{i\in \tilde I_t} ( 1- u_0(X_t^i)) \ge \indic{\text{all particles are greater than $\delta t^{3/2} $ } ; |\tilde I_t| \le e^{2t}} \times (1- C\exp(- c \delta t^{3/2} ))^{e^{2t}}
$$
thanks to our assumption on the behaviour at $+ \infty$ of $u_0$. Therefore, 
$$
u(t,x,\theta) \le 1- (1+ o(1))  \P^{x, \theta}(\text{all particles are greater than $\delta t^{3/2} $ } ; |\tilde I_t| \le e^{2t})
$$
The probability of the event above tends to one by the above observations and Markov's inequality as $\E(|\tilde I_t| ) \le \E ( | I_t| ) = e^t$.

Lower bound. Note that for all $\delta>0$, we can find $A>0$ chosen sufficiently large so that $f(x) \ge 1- \delta$ if $x \le - A$. 
Therefore,
$$
\prod_{i \in \tilde I_t} ( 1- f(X_t^i)) \le \delta + \indic{X_t^i \ge - A \text{ for all } i \in \tilde I_t}
$$
and it follows that 
$
u(t,x, \theta) \ge 1- \delta - \P_{(x, \theta)}( X_t^i \ge - A \text{ for all } i \in \tilde I_t).
$
Fix $x \le x_0 = \gamma t^{3/2}$ with $\gamma < \gamma_0$, and $\theta = 3at /2$. Since $$\P^{(0, \theta)} ( \exists i: X_t^i \ge x +A ) \ge \P^{(0, \theta)}  ( \exists i: X_t^i \ge x_0 +A ) \to 1$$ as $t \to \infty$ (uniformly in $x \le x_0 = \gamma t^{3/2}$) by using \eqref{LB_heavyside} with $x = \gamma't^{3/2}$ where $\gamma < \gamma'< \gamma_0$, we deduce
$$
\liminf_{t\to \infty} \inf_{x\le \gamma t^{3/2} }u(t,x, \theta) \ge 1 - \delta
$$
for $x = \gamma t^{3/2}$, $\theta = 3at/2$. Thus $\inf_{x \le \gamma t^{3/2}} S(t,x) \to 1$ as $t \to \infty$, as desired.
\end{proof}

\section{Global-in-time estimate}
\label{sec:comp-(II)}



\begin{theorem}
  \label{theo:unif-pointwise-II}
  Let us consider a solution $v$ to (NLoc) with initial data
  $v_0$ satisfying the compact support in $\theta$ assumption (1) and the regularity assumption (3), described in Subsection~\ref{sec:main-results}. Then the unique global non-negative $L^\infty$ solution satisfies the global pointwise
  bound 
  \begin{align}
    \label{eq:3}
    \fa t \ge 0, \ x \in \R, \ \theta \ge 1, \quad 0 \le
    w(t,x,\theta) \le M
  \end{align}
  for an explicit constant $M$ independent of the solution. 
\end{theorem}

\begin{lemma}\label{lem:sursol}
Let $v=v(t,x,\theta)$ the solution of (NLoc) with an initial condition satisfying the compact support in $\theta$ and regularity assumptions of Subsection~\ref{sec:main-results}. Then, for any $\tau>0$, there exist $C>0$ such that for $k,\,l\in\mathbb N$, $k+l\leq 3$,
\[\forall (t,x,\theta)\in [0,\tau]\times \mathbb R\times[1,\infty),\quad |\partial_x^k\partial_\theta^l v(t,x,\theta)|\leq Ce^{\frac{\theta^2}{2\tau+1}},\]
and
\[\forall (t,x,\theta)\in [0,\tau]\times \mathbb R\times[1,\infty),\quad |\partial_t v(t,x,\theta)|+|\partial_x\partial_t v(t,x,\theta)|+|\partial_\theta\partial_t v(t,x,\theta)|\leq  Ce^{-\frac{\theta^2}{2\tau+1}}.\]
\end{lemma}

\begin{proof}[Proof of Lemma~\ref{lem:sursol}]

In order to avoid dealing with boundaries in $\theta$, we shift
$[1,+\infty)$ to $[0,+\infty)$ and then do mirror symmetry, i.e. we
define $\bar v(t,x,\theta) := v(t,x,|\theta|+1)$, which solves
\begin{align}
  \label{model-2-sym}
  \left\{ 
  \begin{array}{l} \ds
    \partial_t \bar v = \frac{|\theta| +1}2\Delta_x \bar v + \frac 12\Delta_\theta
    \bar v + \bar v \left( 1-
    \langle \bar v \rangle \right) \\[3mm] \ds
    \bar v = \bar v(t,x,\theta), \ t \ge 0, \ x \in \R, \ \theta \in
    \R, \quad \bar v(t,x,-\theta)=\bar v(t,x,\theta) \\[3mm] \ds
    \langle \bar v \rangle(t,x,\theta) := \int_{\min(\theta- A,0)}
    ^{\theta+A} \bar v(t,x,\omega) \dd
    \omega \dd y, \ \theta \ge 0,\\[4mm] \ds
    \bar v(0,x,\theta) = \bar v_0(x,\theta) \ge 0.
  \end{array}
\right.
\end{align}

Let $\varepsilon>0$, and $\varphi$ a regular approximation of $\theta\mapsto\frac{|\theta|+1}2$. More precisely, we assume that $\varphi(\theta)=\varphi(-\theta)$ for $\theta\in\mathbb R$,  $\varphi(\theta)=\frac{|\theta|+1}2$ for $|\theta|\geq \varepsilon>0$, and that for some $C>0$,
\[\|\varphi'\|_{L^\infty}\leq C,\quad \|\varphi''\|_{L^\infty}\leq \frac C{\varepsilon},\quad\|\varphi'''\|_{L^\infty}\leq \frac C{\varepsilon^2}.\]
Such functions can be constructed for any $\varepsilon>0$, for instance through the rescaling of such a function for $\varepsilon=1$. Let $w$ a solution of 
\begin{equation}\label{eq:homogene0}
\partial_t w-\varphi(\theta)\Delta_x w-\frac 12\Delta_\theta w=w\left(1-\langle w\rangle\right),
\end{equation}
with a regular initial condition $w_0=w_0(x,\theta)$ (in the sense that $\|w_0\|_{C^3(\mathbb R^2)}<\infty$) with a compact support in $\theta$ (that is $w_0(x,\theta)=0$  if $|\theta|>\bar\theta$, for some $\bar\theta>0$). Then, the comparison principle shows that $w(t,x,\theta)\leq \|w_0\|_{L^\infty}e^t$ for $t\geq 0$. More precisely, we notice that
\begin{equation}\label{eq:homogene0bis}
\partial_t w-\varphi(\theta)\Delta_x w-\frac 12\Delta_\theta w-w\leq 0,
\end{equation}
while $\bar w_1(t,x,\theta):= Ce^{t-\frac{\theta^2}{2t+1}}$ is a solution of \eqref{eq:homogene0bis}. If $C>0$ is large enough, $w(0,\cdot,\cdot)\leq \bar w_1(0,\cdot,\cdot)$, and the comparison principle then implies that $0\leq w\leq \bar w$. We notice that for $k\in\{1,2,3\}$, $\partial_x^k w$ satisfies \eqref{eq:homogene0}, the same argument then shows that for some $C>0$, $|\partial_x^k w|\leq Ce^{t-\frac{\theta^2}{2t+1}}$. We notice next that $\partial_\theta w$ satisfies:
\begin{eqnarray}
\partial_t \left(\partial_\theta w\right)-\varphi(\theta)\Delta_x \left(\partial_\theta w\right)-\frac 12\Delta_\theta \left(\partial_\theta w\right)-\partial_\theta w&=&\varphi'(\theta)\Delta_x w+\mathcal O(1)\|w\|_{L^\infty}w\nonumber\\
&\leq& C\left(\|\varphi'\|_{L^\infty}+1\right)e^{t-\frac{\theta^2}{2t+1}}.\label{eq:homogene1}
\end{eqnarray}
Let $\bar w_1(t,x,\theta):=Ce^{\lambda_1 t-\frac{\theta^2}{2t+1}}$, which satisfies 
\[\partial_t \bar w_1-\varphi(\theta)\Delta_x \bar w_1-\frac 12\Delta_\theta \bar w_1-\bar w_1\geq C(\lambda_1-1) e^{t-\frac{\theta^2}{2t+1}}
,\]
so that $\bar w_1$ is a super-solution of \eqref{eq:homogene1} for $t\in[0,\tau]$, as soon as $\lambda_1>0$ is chosen large enough. Just as above, for $k\in\{1,2\}$, it is possible to repeat this method to show that for some $C>0$ and $\lambda_1>0$, $|\partial_x^k\partial_\theta w|\leq Ce^{\lambda_1 t-\frac{\theta^2}{2t+1}}$.
We now turn to $\partial_\theta^2 w$, which satisfies:
\begin{align}
&\partial_t \left(\partial_\theta^2 w\right)-\varphi(\theta)\Delta_x \left(\partial_\theta^2 w\right)-\frac 12\Delta_\theta \left(\partial_\theta^2 w\right)\nonumber\\
&\quad=\varphi''(\theta)\Delta_x w+2\varphi'(\theta)\Delta_x \partial_\theta w+\mathcal O(1)\left(\|\partial_\theta w\|_{L^\infty}w+\|w\|_{L^\infty}\partial_\theta w\right)\nonumber\\
&\quad \leq C\left(\indic{\theta\in[-\varepsilon,\varepsilon]}\frac 1\varepsilon+1\right)e^{\lambda_1 t-\frac{\theta^2}{2t+1}}.\label{eq:homogene2}
\end{align}

Let $\bar w_2(t,x,\theta):=Ce^{\lambda_2 t-\frac{\psi_2(\theta)}{2t+1}}$, where $\psi_2(\theta)=\frac 1\varepsilon \theta^2+1-\varepsilon$ on $[-\varepsilon,\varepsilon]$, and $\psi_2(\theta)=\left(|\theta|+1-\varepsilon\right)^2$ for $|\theta|\geq\varepsilon$. Then, $\psi_2\in C^1(\mathbb R)$ (in particular, $\|\psi'\|_{L^\infty}<2$), and $\psi_2''(\theta)=2 \left(\frac 1\varepsilon\indic{\theta\in[-\varepsilon,\varepsilon]}+\indic{\theta\in[-\varepsilon,\varepsilon]^c}\right)$. Then, for $t\in[0,\tau]$,
\begin{eqnarray*}
\partial_t \bar w_2-\varphi(\theta)\Delta_x \bar w_2-\frac 12\Delta_\theta \bar w_2&=&\left(\lambda_2 +\frac{2\psi_2(\theta)}{(2t+1)^2}+\frac{\psi_2''(\theta)}{4t+2}-\frac{(\psi_2'(\theta))^2}{2(2t+1)^2}\right)\bar w_2\\
&=&\left(\lambda_2+\frac{2\psi_2(\theta)}{(2t+1)^2}+\frac{\frac 1\varepsilon\indic{\theta\in[-\varepsilon,\varepsilon]}+\indic{\theta\in[-\varepsilon,\varepsilon]^c}}{2t+1}-\frac{(\psi_2'(\theta))^2}{2(2t+1)^2}\right)\bar w_2\\
&\geq&\left((\lambda_2-2)+\frac {1}\varepsilon \indic{\theta\in[-\varepsilon,\varepsilon]}\right)Ce^{\lambda_2 t-\frac{\psi_2(\theta)}{2t+1}}.
\end{eqnarray*}
Since $\psi_2(\theta)\leq \theta^2+1$, if we chose $C,\,\lambda_2>0$ large enough, then $\bar w_2$ is a super solution of \eqref{eq:homogene2} for  $t\in[0,\tau]$, and then $|\partial_\theta^2 w|\leq \bar w_2\leq Ce^{\lambda_2 t-\frac{\theta^2}{2t+1}}$ for some $C>0$ (for $t\in[0,\tau]$). Here also, a similar estimate applies to $\partial_x\partial_\theta^2 w$, to show that $|\partial_x\partial_\theta^2 w|\leq Ce^{\lambda_2 t-\frac{\theta^2}{2t+1}}$. Note that this estimate is sufficient to prove the well posedness of the problem for $t\in[0,\tau]$, and thus in particular the symmetry of $w$ in $\theta$: $w(t,x,\theta)=w(t,x,-\theta)$ for $(t,x,\theta)\in[0,\tau]\times\mathbb R^2$, which implies in particular that $\partial_\theta^3 w(t,x,0)=0$ for $(t,x)\in\mathbb R_+\times\mathbb R$. This property will be important to estimate $\partial_\theta^3 w$. We have
\begin{align}
&\partial_t \left(\partial_\theta^3 w\right)-\varphi(\theta)\Delta_x \left(\partial_\theta^3 w\right)-\frac 12\Delta_\theta \left(\partial_\theta^3 w\right)\nonumber\\
&\quad=\varphi'''(\theta)\Delta_x w+3\varphi''(\theta)\Delta_x \partial_\theta w+3\varphi'(\theta)\Delta_x \partial_\theta^2 w\nonumber\\
&\qquad+\mathcal O(1)\left(\|w\|_{L^\infty}+\|\partial_\theta w\|_{L^\infty}+\|\partial_\theta^2 w\|_{L^\infty}\right)\left(w+\partial_\theta w+\partial_\theta^2 w\right)\nonumber\\
&\quad\leq C \indic{\theta\in[-\varepsilon,\varepsilon]}\left(|\varphi'''(\theta)|+|\varphi''(\theta)|\right)+C\|\varphi'\|_{L^\infty}e^{\lambda_1^2 t-\frac{\theta^2}{2t+1}}\nonumber\\
&\quad\leq C\left(\frac 1{\varepsilon^2}\indic{\theta\in[-\varepsilon,\varepsilon]}+1\right)e^{\lambda_1^2 t-\frac{\theta^2}{2t+1}}.
\label{eq:homogene3}
\end{align}
Since $\partial_\theta^3 w(t,x,0)=0$ for $\theta=0$, $\partial_\theta^3 w(t,x,\theta)$ is a sub-solution of the problem\eqref{eq:homogene3} on $\mathbb R_+\times\mathbb R\times [0,\infty)$ with a Dirichlet boundary condition in $\theta=0$. We will build a super-solution for this half-space problem:
\[\bar w_2(t,x,\theta)=\min\left(C_1e^{\lambda_2 t-\frac{\theta^2}{2t+1}},C_2\left(1+\sqrt{\theta/\varepsilon}\right)\right)=\indic{\theta\in [\bar\theta(t),\infty)}C_1e^{\lambda_2 t-\frac{\theta^2}{2t+1}}+\indic{\theta\in (0,\bar\theta(t))}C_2\sqrt{\theta/\varepsilon},\]
for $t\in[0,\tau]$. Moreover, if $C_1$ is chosen sufficiently larger than $C_2$, then $\varepsilon<\theta(t)<C\varepsilon$, for some constant $C>0$. Since a minimum of two super-solutions is a super-solution we simply need to check that $\bar w_{2,1}(t,x,\theta)=C_1e^{\lambda_2 t}$ is a super-solution of \eqref{eq:homogene3} on $\mathbb R_+\times\mathbb R\times[\varepsilon,\infty)$, and $\bar w_{2,2}(t,x,\theta)=C_2\left(1+\sqrt{\theta/\varepsilon}\right)$ is a super-solution of \eqref{eq:homogene3} on $\mathbb R_+\times\mathbb R\times(0,C\varepsilon)$. The argument for $\bar w_{2,1}$ is similar to earlier cases, while $\bar w_{2,2}$ satisfies
\[\partial_t \bar w_{2,2}-\varphi(\theta)\Delta_x \bar w_{2,2}-\frac 12\Delta_\theta \bar w_{2,2}-\bar w_{2,2}=\frac{C_2}{4\sqrt\varepsilon}\theta^{-3/2}-C_2\left(1+\sqrt{\theta/\varepsilon}\right)\geq \frac {C_2}{8\varepsilon^2},\]
provided $\varepsilon>0$ is small enough. $\bar w_{2,2}$ is thus indeed a super-solution of \eqref{eq:homogene3} on $[0,\tau]\times\mathbb R\times(0,C\varepsilon)$, provided $C_2$ is chosen large enough. The comparison principle then shows that $|\partial_\theta^3 w(t,x,\theta)|\leq \bar w_2(t,x,|\theta|)\leq Ce^{\lambda_2 t-\frac{\theta^2}{2t+1}}$.

\medskip

We notice next that $\partial_t w(0,\cdot,\cdot)=\varphi(\theta)\Delta_x w(0,\cdot,\cdot)-\frac 12\Delta_\theta w(0,\cdot,\cdot)-w(0,\cdot,\cdot)$ is in $C^1(\mathbb R)$ with a compact support in $\theta$, and $\partial_t w$ is a solution of an equation similar to \eqref{eq:homogene0}. The argument above (see \eqref{eq:homogene1} and \eqref{eq:homogene1}) can then be reproduced to show that $|\partial_t w(t,x,\theta)|+|\partial_x\partial_t w(t,x,\theta)|+|\partial_\theta\partial_t w(t,x,\theta)|\leq  Ce^{\lambda_t t-\frac{\theta^2}{2t+1}}$.

We have proven that there exists $C>0$ independent of $\varepsilon>0$ (small enough) such that for $k,\,l\in\mathbb N$, $k+l\leq 3$,
\[\forall (t,x,\theta)\in [0,\tau]\times \mathbb R^2,\quad |\partial_x^k\partial_\theta^l w(t,x,\theta)|\leq Ce^{-\frac{\theta^2}{2\tau+1}},\]
and
\[|\partial_t w(t,x,\theta)|+|\partial_x\partial_t w(t,x,\theta)|+|\partial_\theta\partial_t w(t,x,\theta)|\leq  Ce^{-\frac{\theta^2}{2\tau+1}}.\]
We can thus pass to the limit $\varepsilon\to 0$ in these estimates to obtain similar estimates on solutions of \eqref{model-2-sym}, which conclude the proof.

%


\end{proof}

We will now prove Theorem~\ref{theo:unif-pointwise-II}. Note that this proof draws a lot of inspiration from \cite{Turanova}. 

\begin{proof}[Proof of Theorem~\ref{theo:unif-pointwise-II}]
\mbox{ } 

\noindent
\emph{Step~1. Definitions and rescaling.}
We define the cylinder around a point $\bar z := (\bar t, \bar x, \bar
\theta)$:
\begin{equation}\label{def:Qr}
Q_{\bar z,R}:=(\bar t-R^2,\bar t)\times B((\bar x,\bar\theta),R).
\end{equation}
We might omit the base point and/or the size $R$ when obvious from the
context. We define for any cylinder $Q=Q_{\bar z,R}$ the norms
\[[u]_{\delta/2,\delta, Q}=\sup_{(t,x)\neq (s,y)\in Q}\frac{|u(t,x)-u(s,y)|}{\left(|x-y|+|t-s|^{1/2}\right)^\delta},\]
\[|u|_{\delta/2,\delta, Q}=\|u\|_{L^\infty(Q)}+[u]_{\delta/2,\delta, Q},\]
\[[u]_{1+\delta/2,2+\delta, Q}=[\partial_t u]_{\delta/2,\delta, Q}+\sum_{i,j=1}^2[\partial_{x_ix_j}^2 u]_{\delta/2,\delta, Q},\]
\begin{eqnarray*}
 |u|_{1+\delta/2,2+\delta, Q}&=&\|u\|_{L^\infty(Q)}+\sum_{i=1}^2\|\partial_{x_i} u\|_{L^\infty(Q)}+\|\partial_{t} u\|_{L^\infty(Q)}\\
&&+\sum_{i,j=1}^2\|\partial_{x_ix_j}^2 u\|_{L^\infty(Q)}+[u]_{1+\delta/2,2+\delta, Q}.
\end{eqnarray*}

We recall the definition \eqref{model-2-sym} of $\bar v$, and introduce an additional notation: for a given base point $\bar z$, we rescale the problem, $\tilde v =
(t,x,\theta) = \bar v\left(t,\sqrt{\bar \theta+1} x, \theta\right)$, to get 
\begin{align}
  \label{model-2-sym-resca}
  \left\{ 
  \begin{array}{l} \ds
    \partial_t \tilde v = \frac{|\theta| +1}{2\left(|\bar \theta|+1\right)} \Delta_x \tilde v + \Delta_\theta
    \tilde v + \tilde v \left( 1-
    \langle \tilde v \rangle \right) \\[3mm] \ds
    \tilde v = \tilde v(t,x,\theta), \ t \ge 0, \ x \in \R, \ \theta \in
    \R, \quad \tilde v(t,x,-\theta)=\tilde v(t,x,\theta) \\[3mm] \ds
    \langle \tilde v \rangle(t,x,\theta) := \int_{\min(\theta- A,0)}
    ^{\theta+A} \tilde v(t,x,\omega) \dd
    \omega \dd y, \ \theta \ge 0,\\[4mm] \ds
    \tilde v(0,x,\theta) = \tilde v_0(x,\theta) \ge 0.
  \end{array}
\right.
\end{align}
\medskip

\emph{Step~2. Relating $|u|_{1+\delta/2,2+\delta, Q}$ to $\|\cdot\|_{L^\infty}$.}

We use (in the particular framework needed here) the following two
results from \cite{Krylov}:
\begin{lemma}[Theorem~8.11.1 in \cite{Krylov}]\label{lem:interior-est}
  Let $\delta\in (0,1)$,
  $(\bar t,\bar x,\bar\theta)\in\R_+\times\R^2$. If
  $S \in C^{\delta/2}_t C^{\delta}_{x,\theta}(Q_2)$ (i.e.
  $|S|_{\delta/2,\delta,Q_2}<\infty$) and $V(t,x,\theta)$ is a solution of
\[\partial_t V-\frac{|\theta|+1}{2\left(|\bar \theta|+1\right)}\Delta_x V-\Delta_\theta V=S,\textrm{ on }Q_2,\] 
then, there exist a universal constant $C>0$ such that
\[|V|_{1+\delta/2,2+\delta,Q_1}\leq C\left(|S|_{\delta/2,\delta,Q_2}+\|V\|_{L^\infty(Q_2)}\right). \]
\end{lemma}
\begin{lemma}[Theorem~8.8.1 in \cite{Krylov}]\label{lem:interpolation} Let
  $(\bar t,\bar x,\bar \theta)\in\R_+\times\R^2$ and $\delta>0$. There
  exists a constant $N>0$ such that for any $\varepsilon>0$, and any
  $V\in C^{1+\delta/2}_t C^{2+\delta}_{x,\theta}(Q_3)$,
  \[[V]_{\delta/2,\delta, Q_2}\leq \varepsilon
  [V]_{1+\delta/2,2+\delta,Q_3}+N\varepsilon^{-\delta/2}\|V\|_{L^\infty(Q_3)}.\]
\end{lemma}  

%


We recall the definition~\eqref{def:Qr} of  $Q_{(\bar t,\bar x,\bar\theta),R}$, and $\tilde v$ as in \eqref{model-2-sym-resca}. Thanks to Lemma~\ref{lem:sursol}, if $\bar t\in[1,(3+A)^2+2]$, then
\begin{eqnarray}
|\tilde v|_{1+\delta/2,2+\delta, Q_{(\bar t,\bar x,\bar \theta),1}}&\leq& C\left(|\bar \theta|^{2+\delta}+1\right)\bigg(\sum_{k+l\leq 3}\|\partial_x^k\partial_\theta^l v(t,x,\theta)\|_{L^\infty([0,\bar t]\times\mathbb R\times [-\bar\theta-A,\bar \theta+A])}\nonumber\\
&&+\|\partial_t v(t,x,\theta)\|_{L^\infty([0,\bar t]\times\mathbb R\times [-\bar\theta-A,\bar \theta+A])}+\|\partial_x\partial_t v(t,x,\theta)\|_{L^\infty([0,\bar t]\times\mathbb R\times [-\bar\theta-A,\bar \theta+A])}\nonumber\\
&&+\|\partial_\theta\partial_t v(t,x,\theta)\|_{L^\infty([0,\bar t]\times\mathbb R\times [-\bar\theta-A,\bar \theta+A])}\bigg)\leq C_0,\label{est-cond-ini}
\end{eqnarray}
where the constant $C_0>0$ is independent of $(\bar t,\bar x,\bar \theta)\in [1,(3+A)^2+2)\times \R^2$.

\medskip

Let $T>(3+A)^2+1$, and $M$ such that
\begin{equation}\label{assM}
M>2C_0,
\end{equation}
where $C$ is here the constant appearing in \eqref{est-cond-ini}. We assume also that $\|v\|_{L^\infty([0,T]\times\mathbb R\times\mathbb R_+)}\leq M$. For $(\bar t,\bar x,\bar \theta)\in [(3+A)^2,T)\times \R^2$, we can apply Lemma~\ref{lem:interior-est} to obtain
\[|\tilde v|_{1+\delta/2,2+\delta,Q_{1}}\leq C\left(|\left(1-\langle \tilde v\rangle\right)\tilde
  v|_{\delta/2,\delta,Q_{2}}+\|\tilde v\|_{L^\infty(Q_{2})}\right). \]
We estimate further
\begin{align*}
 &|\left(1-\langle \tilde v\rangle\right)\tilde v|_{\delta/2,\delta,Q_{2}\cap (\R_+\times\R^2)}\\
&\quad=[\left(1-\langle \tilde v\rangle\right)\tilde v]_{\delta/2,\delta,Q_{2}\cap( \R_+\times\R^2)}+\|\left(1-\langle \tilde v\rangle\right)\tilde v\|_{L^\infty(Q_{2}\cap (\R_+\times\R^2))}\\
&\quad\leq\|1-\langle \tilde v\rangle\|_{L^\infty(Q_{2}\cap (\R_+\times\R^2))}[\tilde v]_{\delta/2,\delta,Q_{2}\cap \R_+\times\R^2}\\
&\quad+[1-\langle \tilde v\rangle]_{\delta/2,\delta,Q_{2}\cap (\R_+\times\R^2)}\| \tilde v\|_{L^\infty(Q_{2}\cap (\R_+\times\R^2))}\\
&\qquad+\|1-\langle \tilde v\rangle\|_{L^\infty(Q_{2}\cap (\R_+\times\R^2))}\|\tilde v\|_{L^\infty(Q_{2}\cap (\R_+\times\R^2))}\\
&\quad\leq C\left(M[\tilde v]_{\delta/2,\delta,Q_{2+A}\cap (\R_+\times\R^2)}+M^2\right),
\end{align*}
and then,
\begin{equation*}
|\tilde v|_{1+\delta/2,2+\delta,Q_{1}\cap (\R_+\times\R^2)}\leq C M\left([\tilde v]_{\delta/2,\delta,Q_{2+A}\cap (\R_+\times\R^2)}+M\right).
\end{equation*}
Using Lemma~\ref{lem:interpolation}, we get
\[|\tilde v|_{\delta/2,\delta,Q_{2+A}}\leq
CM\left(\varepsilon [\tilde
  v]_{1+\delta/2,2+\delta,Q_{3+A}}+\left(1+\varepsilon^{-\delta/2}\right)M \right).\]
We select $\varepsilon \sim \alpha M^{-1}$, to get
\begin{equation}\label{eq-step2}
|\tilde v|_{1+\delta/2,2+\delta,Q_{1}}\leq \alpha|\tilde
v|_{1+\delta/2,2+\delta,Q_{3+A}} + C M^{2+\delta/2},
\end{equation}
for some new $C>0$. 

\medskip

Let now  $\tilde v,\,\tilde v_*$ as in \eqref{model-2-sym-resca}, for $(\bar t,\bar x,\bar \theta)$ and $(\bar t_*,\bar x_*,\bar \theta_*)$ respectively, where $|\bar\theta-\bar \theta_*| \leq 3+A$. Let also $r>0$. Then, 
\[\tilde v(t,x,\theta)=\bar v\left(t,\sqrt{|\bar\theta|+1}x,\theta\right)=\tilde v_*\left(t,\sqrt{\frac{|\bar\theta|+1}{|\bar \theta_*|+1}}x,\theta\right).\] 
Let $\phi_{(\bar t_*,\bar x_*,\bar \theta_*)}:(t,x,\theta)\mapsto \left(t,\sqrt{\frac{|\bar\theta|+1}{|\bar \theta_*|+1}}x,\theta\right)$. 
Since we assumed that $|\bar\theta-\bar \theta_*| \leq 3+A$, we have $\sqrt{\frac{|\bar\theta|+1}{|\bar \theta_*|+1}}\leq 3+A$, which implies
\[|\tilde v|_{1+\delta/2,2+\delta,Q_{(\bar t_*,\bar x_*,\bar \theta_*),r}}\leq C|\tilde v_*|_{1+\delta/2,2+\delta,\phi_{(\bar t_*,\bar x_*,\bar \theta_*)}\left(Q_{(\bar t_*,\bar x_*,\bar \theta_*),r}\right)}.\]
Moreover, $\phi_{(\bar t_*,\bar x_*,\bar \theta_*)}\left(Q_{(\bar t_*,\bar x_*,\bar \theta_*),r}\right)\subset Q_{\phi_{(\bar t_*,\bar x_*,\bar \theta_*)}(\bar t_*,\bar x_*,\bar \theta_*),(3+A)r}$, and then,
\[|\tilde v|_{1+\delta/2,2+\delta,Q_{(\bar t_*,\bar x_*,\bar \theta_*),r}}
\leq C|\tilde v_*|_{1+\delta/2,2+\delta,Q_{\phi_{(\bar t_*,\bar x_*,\bar \theta_*)}(\bar t_*,\bar x_*,\bar \theta_*),(3+A)r}}.\]
We notice now that there exists $\left((\bar t_i,\bar x_i,\bar \theta_i)\right)_{i=1,\dots,N}$,  (where $N$ is a function of $A$ only, and $\bar t_i\in (1,\bar t]$) such that $Q_{(\bar t,\bar x,\bar \theta),3+A}\subset \cup_{i=1}^N Q_{(\bar t_i,\bar x_i,\bar \theta_i),1/(3+A)}$. Thus,
\begin{eqnarray}
|\tilde v|_{1+\delta/2,2+\delta,Q_{3+A}}&\leq& C\sum_{i=1}^N |\tilde v|_{1+\delta/2,2+\delta,Q_{(\bar t_i,\bar x_i,\bar \theta_i),1/(3+A)}}\nonumber\\
&\leq&C\sum_{i=1}^N |\tilde v_i|_{1+\delta/2,2+\delta,Q_{\phi_{(\bar t_i,\bar x_i,\bar \theta_i)}(\bar t_i,\bar x_i,\bar \theta_i),1}},\label{eq-step3}
\end{eqnarray}
where $\tilde v_i$ is the equivalent of $\tilde v$, with $(\bar t_i,\bar x_i,\bar \theta_i)$ instead of $(\bar t,\bar x,\bar \theta)$.

\medskip

 We define now, for some $T>0$,
\[\|v\|_{1+\delta/2,2+\delta,T}:=\max_{(\bar t,\bar x,\bar \theta)\in [1,\bar t]\times \mathbb R^2} |\tilde v|_{1+\delta/2,2+\delta,Q_{(\bar t,\bar x,\bar \theta),1}}.\]
If $\bar t\in [(3+A)^2+1,T]$, we can apply \eqref{eq-step2} and \eqref{eq-step3} to show that
\begin{eqnarray*}
|\tilde v|_{1+\delta/2,2+\delta,Q_{1}}&\leq& \alpha  |\tilde
v|_{1+\delta/2,2+\delta,Q_{3+A}} + C M^{2+\delta/2}\\
&\leq& \alpha  C\sum_{i=1}^N |\tilde v_i|_{1+\delta/2,2+\delta,Q_{\phi_{(\bar t_i,\bar x_i,\bar \theta_i)}(\bar t_i,\bar x_i,\bar \theta_i),1}} + C M^{2+\delta/2}\\
&\leq& \alpha  CN \|v\|_{1+\delta/2,2+\delta,T}+ C M^{2+\delta/2}.
\end{eqnarray*}
This estimate holds for any $(\bar t, \bar x,\bar \theta)\in[ (3+A)^2+1,T]\times\mathbb R^2$, and thanks to the assumption \eqref{assM} on $M$ and \eqref{est-cond-ini}, it also holds for $(\bar t, \bar x,\bar \theta)\in[0, (3+A)^2+1]\times\mathbb R^2$. Thus,
\[\|v\|_{1+\delta/2,2+\delta,T}\leq \alpha CN\|v\|_{1+\delta/2,2+\delta,T}+ C M^{2+\delta/2},\]
and we chose $\alpha:= \frac 1{2CN}$ to obtain that
\[\|v\|_{1+\delta/2,2+\delta,T}\leq C M^{2+\delta/2}.\]

\medskip

\noindent
\emph{Step 3. Maximum principle.}  Thanks to \eqref{assM} and \eqref{est-cond-ini}, we know that $\|v\|_{L^\infty([0,T]\times\mathbb R\times\mathbb R_+)}\leq M$. Our goal is to show that indeed, $\|v\|_{L^\infty([0,T]\times\mathbb R\times\mathbb R_+)}< M$.

Assume that there exists  
$(\bar t,\bar x,\bar \theta)\in (0,\infty)\times\mathbb R\times
[1,\infty)$ such that $v$ reaches the value $M$. Then 
$v(\bar t,\bar x,\bar \theta+1)=M$, while
\[\forall (t,x,\theta)\in [0,\bar t]\times \mathbb R\times [1,\infty),\quad v(t,x,\theta)\leq M.\]
If we define as before $\bar v(t,x,\theta):=v(t,x,|\theta|+1)$ (see \eqref{model-2-sym}), then
\[\partial_t \bar v(\bar t,\bar x,\bar \theta)\geq 0,\; \Delta_x \bar v(\bar t,\bar x,\bar \theta)\leq 0,\;\Delta_\theta \bar v(\bar t,\bar x,\bar \theta)\leq 0,\]
which, combined to \eqref{model-2-sym}, implies
\[0\leq v(\bar t,\bar x,\bar \theta)\left(1-\langle v\rangle(\bar t,\bar x,\bar \theta)\right),\]
and since $v(\bar t,\bar x,\bar \theta)>0$,
\begin{equation}\label{eq:ineq1}
0\leq 1-\langle v\rangle(\bar t,\bar x,\bar \theta).
\end{equation}
For any $a\in(0,A)$ 
\begin{eqnarray}
\langle \bar v\rangle (\bar t,\bar x,\bar \theta)&\geq&\frac 12\int_{\bar\theta-a}^{\bar\theta+a}v(t,x,\theta')\,d\theta'\nonumber\\
&\geq& a v(\bar t,\bar x,\bar \theta)-a^2\|\partial_\theta v(\bar t,\bar x,\cdot)\|_{L^\infty([\bar\theta-a,\bar\theta+a])}.\label{eq:ineq2}
\end{eqnarray}
We can now use the Gagliardo-Nirenberg interpolation inequality and
the previous step with $T:=\bar t$, to estimate the last term of \eqref{eq:ineq2}:
\begin{eqnarray*}
\|\partial_\theta \bar v(\bar t,\bar x,\cdot)\|_{L^\infty([\bar\theta-a,\bar\theta+a])}&\leq& \|\partial_\theta \bar v(\bar t,\bar x,\cdot)\|_{L^\infty([\bar\theta-a,\bar\theta+a])}\\
&\leq&C \|\bar v\|_{L^\infty([\bar\theta-a,\bar\theta+a])}^{\frac 12}\|\Delta_{\theta}\bar v\|_{L^\infty([\bar\theta-a,\bar\theta+a])}^{\frac 12}+C\|\bar v\|_{L^\infty([\bar\theta-a,\bar\theta+a])}\nonumber\\
&\leq&CM^{\frac 12}\left(CM^{2+\delta/2}\right)^{\frac 12}+CM=CM^{\frac {3+\delta/2}2},
\end{eqnarray*}
where $C>0$ is a universal constant. Thanks to \eqref{eq:ineq1}, \eqref{eq:ineq2} and the last estimate, we get
\[0\leq 1-\left(aM-Ca^2M^{\frac {3+\delta/2}2}\right).\]
If we select $a:=2M^{-1}$, we get
\[0\leq 1-\left(2-\frac {4C}{M^{\frac {1-\delta /2}2}}\right),\]
in which we chose e.g. $\delta=1/2$. We then obtain a contradiction as soon as $M>(4 C)^4$. This contradiction implies that $\|v\|_{L^\infty([0,T]\times\mathbb R\times\mathbb R_+)}< M$.

Let now a constant $M$ satisfying \eqref{assM}. For $T\in((3+A)^2+1,(3+A)^2+2)$, we have $\|v\|_{L^\infty([0,T]\times\mathbb R\times\mathbb R_+)}< M$. We can then define the largest $T>(3+A)^2+1$ such that $\|v\|_{L^\infty([0,T]\times\mathbb R\times\mathbb R_+)}\leq M$. If $T\neq\infty$, the argument above leads to a contradiction, which proved the Theorem.
\end{proof}

\section{Comparison of the models}
\label{sec:comparison-models}

We show two comparison principles between the models (Loc) and (NLoc). First, we construct a solution of (Loc) that will provide a lower bound for solutions of (NLoc): 
\begin{proposition}\label{prop:nonlocal}
Let $v_0\in C^{2+\delta}(\mathbb R\times [1,\infty))$ with compact support in $\theta$, thin tail in $x$ and regular, as described in Subsection~\ref{sec:main-results}. Let $v(t,x,\theta)$ the corresponding solution of (NLoc). For any $\eta>0$ small, there exists $\varepsilon>0$ such that
For any $\eta>0$, there exists $\varepsilon>0$ such that 
\[\forall (t,x,\theta)\in\mathbb R_+\times \mathbb R\times [1,\infty),\quad \varepsilon u\left((1-\eta)t,\sqrt{1-\eta}\,x,\sqrt{1-\eta}(\theta-1)+1\right)\leq v(t,x,\theta),\]
where $u$ is the solution of (Loc) with initial condition 
\begin{equation}\label{condini-comp-inf}
u_0(t,x,\theta)=\indic{\mathbb R_-\times (1,2)}.
\end{equation}

\end{proposition}

\begin{proof}[Proof of Proposition~\ref{prop:nonlocal}]

Let $\eta>0$. We recall the definition \eqref{model-2-sym} of $\bar v$.   Thanks to Theorem~\ref{theo:unif-pointwise-II}, $\bar v$ satisfies
\[\partial_t \bar v-\frac{1+|\theta|}2\Delta_x \bar v-\frac 12\Delta_\theta \bar v\geq (1-2AM)\bar v,\]
and since $\bar v(0,x,\theta)>C>0$ for $(x,\theta)\in\mathbb R_-\times (\theta_{\min},\theta_{\max})$, there exists $C_0>0$ such that $\bar v(1,x,\theta)>C_0$ for $(x,\theta)\in\mathbb R_-\times [-1,1]$. Then,
\begin{equation}\label{ineq-ini}
C_0 \bar u_0(0,x,\theta)<v_0(1,x,\theta),\textrm{ for }(x,\theta)\in\mathbb R_-\times [-1,1].
\end{equation}
For any $(x,\bar\theta)\in [1,\infty)\times\mathbb R$, $\bar v(t,x,\theta)=\tilde v\left(t,\sqrt{\bar \theta+1},\theta\right)$ is  solution  of \eqref{model-2-sym-resca}, which is a parabolic equation, and the coefficients of this equation are bounded for $(t,x,\theta)\in \mathbb R\times \mathbb R\times [\bar \theta-2A,\bar \theta +2A]$, with a bound on those coefficients that is uniform in $(\bar x,\bar \theta)$. Thanks to this property and the $L^\infty$ bound $\|\bar v\|_{L^\infty}\leq M<\infty$ provided by Theorem~\ref{theo:unif-pointwise-II}, we can apply the Harnack-type inequality Theorem~2.6 from \cite{Alfaro} with $\delta:=\frac\eta 2$. There exists then  $C_H>0$ such that for any $t\geq 1$, 
\[\langle \tilde v\rangle (t,\bar x,\bar \theta)\leq C_H\tilde v(t,\bar x,\theta)+\frac\eta 2,\]
and then $\langle \bar v\rangle (t,\bar x,\bar \theta)\leq C_H\bar v(t,\bar x,\bar \theta)+\frac\eta 2$, where the constant $C_H$ is independent of $(\bar x,\bar\theta)$. $\bar v$ then satisfies
\begin{equation}\label{ineq-Harnack}
\partial_t\bar v -\frac{|\theta|+1} 2\Delta_x\bar v-\frac 12\Delta_\theta\bar v\geq \left(1-\frac \eta 2-2AC_H\bar v\right)\bar v.
\end{equation}
Let now $\bar u(t,x,\theta):= u(t,x,|\theta|+1)$, and notice that $\hat u(t,x,v):=\varepsilon u\left((1-\eta)t,\sqrt{1-\eta}\,x,\sqrt{1-\eta}\,(\theta-1)+1\right)$ satisfies
\begin{equation*}
\partial_t\hat u -\frac{|\theta|+1} 2\Delta_x\hat u-\frac 12\Delta_\theta\hat u= (1-\eta)\left(1-\frac 1\varepsilon\hat u\right)\hat u.
\end{equation*}
If we chose $\varepsilon=\min(4AC_H,C_0)$, then $\hat u$ is a sub-solution of \eqref{ineq-Harnack}, which, combined to \eqref{ineq-ini} and the comparison principle, implies that for $(t,x,\theta)\in\mathbb R_+\times \mathbb R\times [1,\infty)$,
\[\varepsilon u\left((1-\eta)t,\sqrt{1-\eta}\,x,\sqrt{1-\eta}\,(\theta-1)+1\right)\leq v(t,x,\theta).\]

\end{proof}

The second step is to construct a solution of (Loc) that will provide an upper bound for solutions of (NLoc): 
\begin{proposition}\label{prop:nonlocal2}
Let $v$ a solution of (NLoc) such that its initial condition $v_0$ satisfies (1) and (2) of Subsection~\ref{sec:main-results}. For any $\eta>0$ small, there exists $\varepsilon>0$ such that 
\[\forall (t,x,\theta)\in\mathbb R_+\times \mathbb R\times [1,\infty),\quad v(t,x,\theta)\leq  \frac 1\varepsilon u\left((1+\eta)t,\sqrt{1+\eta}\,x,\sqrt{1+\eta}\,\theta\right),\]
where $u$ the solution of (Loc) with initial condition 
\begin{equation}\label{cond_ini_u}
u_0(x,\theta)=\left\{\begin{array}{l}\varepsilon v_0\left(\frac x{\sqrt{1+\eta}},\frac{\theta-1}{\sqrt{1+\eta} }+1\right),\textrm{ for }(x,\theta)\in \left(\mathbb R_-\times (1,2)\right)^c\\
1,\textrm{ for }(x,\theta)\in \mathbb R_-\times (1,2). \end{array}\right.
\end{equation}
\end{proposition}

\begin{proof}[Proof of Proposition~\ref{prop:nonlocal2}]
We notice that $v$ satisfies
\[
\partial_t v-\frac \theta 2\Delta_x v-\frac 12 \Delta_\theta v=v(1-\langle v\rangle )\leq v< v\left(1+\eta-\frac\eta {2M}v\right).
\]
If $u$ is a solution of (Loc), then $\hat u=\frac{4M}\eta u\left((1+\eta)t,\sqrt{1+\eta}x,\sqrt{1+\eta}(\theta-1)+1\right)$ satisfies
\begin{equation}\label{eqsup}
\partial_t \hat u-\frac \theta 2\Delta_x \hat u-\frac 12 \Delta_\theta \hat u=\hat u\left(1+\eta-\frac\eta {2M}\hat u\right).
\end{equation}
Moreover, if we assume that the initial condition of $u$ is given by \eqref{cond_ini_u} with $\varepsilon=\frac \eta{2M}$, then $v_0\leq \hat u(0,\cdot,\cdot)$ (notice that $\|v_0\|_{L^\infty}\leq M<\frac{4M}\eta$), and the comparison principle applied to the (local) parabolic equation \eqref{eqsup} implies that $v(t,x,\theta)\leq \hat u(t,x,\theta)$ for $(t,x,\theta)\in\mathbb R_+\times\mathbb R\times[1,\infty)$, which proves the result.

\end{proof}

\begin{proof}[Proof of Theorem~\ref{T:toads}]

Let us first consider the upper bound \eqref{toads:UB-nonlocal} on the propagation of $v$. Thanks to Proposition~\ref{prop:nonlocal2}, for any $\eta>0$ there exists $\varepsilon>0$  such that
\[v\left(\frac t{1+\eta},\frac x{\sqrt{1+\eta}},\frac\theta{\sqrt{1+\eta}}\right)\leq \frac 1\varepsilon u(t,x,\theta),\]
where $u$ is the solution of (Loc) with initial condition \eqref{cond_ini_u}. Since the initial condition \eqref{cond_ini_u} satisfies the conditions (1) and (2') (see Subsection~\ref{sec:key-ideas-proofs}), Theorem~\ref{T:toads} applies, and for any $\tilde \gamma>\gamma_0$, there exists $\tilde \varepsilon>0$ such that
\[
\lim_{t\to\infty} \sup_{\theta\geq 1}v\left(\frac t{1+\eta},\frac {\tilde\gamma}{\sqrt{1+\eta}}t^{\frac 32},\theta\right)=0,
\]
that is $\lim_{t\to\infty} \sup_{\theta\geq 1}v\left(t, \tilde \gamma(1+\eta)t^{\frac 32},\theta\right)=0$, which proves \eqref{toads:UB-nonlocal}, provided we chose $\eta>0$ small enough, and $\tilde \gamma>\gamma_0$ small enough.

Proving the lower bound \eqref{toads:LB-nonlocal} is very similar: Thanks to Proposition~\ref{prop:nonlocal}, for any $\eta>0$ there exists $\varepsilon>0$  such that
\[\varepsilon u(t,x,\theta)\leq v\left(\frac t{1-\eta},\frac x{\sqrt{1-\eta}},\frac \theta{\sqrt{1+\eta}}\right),\]
where $u$ is the solution of (Loc) with initial condition \eqref{condini-comp-inf}.Since the initial condition \eqref{cond_ini_u} satisfies the conditions (1) and (2'), Theorem~\ref{T:toads} applies, and for any $\tilde \gamma<\gamma_0$, there exists $\tilde \varepsilon>0$ such that for any $t\geq 1$,
\[
\tilde\varepsilon \leq  \sup_{\theta\geq 1}v\left(\frac t{1-\eta},\frac {\tilde \gamma}{\sqrt{1-\eta}}t^{\frac 32},\theta\right),
\]
that is, for any $t\geq 2$,
\[
\tilde\varepsilon\leq  \sup_{\theta\geq 1}v\left(t, \tilde \gamma(1-\eta)t^{\frac 32},\theta\right),
\]
which proves \eqref{toads:LB-nonlocal}, provided we chose $\eta>0$ small enough, and $\tilde\gamma<\gamma_0$ large enough.
\end{proof}

\section*{Acknowledgements}

The first author acknowledges the financial support of EPSRC grants EP/L018896/1 and EP/I03372X/1. The second author’s work is supported by the ERC starting grant MATKIT. The third author was partially supported by the ANR grant  MODEVOL, ANR-13-JS01-0009, and by the CNRS/Royal Society exchange project CODYN.

\signnb
\signcm 
\signgr

\end{document}